\documentclass[11pt,a4paper,leqno]{article}
\usepackage{a4wide}
\usepackage{latexsym}
\usepackage{amsmath}
\usepackage{amssymb}
\usepackage{graphicx}
\newtheorem{defin}{Definition}
\newtheorem{lemma}{Lemma}
\newtheorem{prop}{Proposition}
\newtheorem{theo}{Theorem}

\newenvironment{proof}{\medskip\par\noindent{\bf Proof}}{\hfill $\Box$
\medskip\par}
\begin{document}
\title{On parametric Gevrey asymptotics for initial value problems with infinite order irregular singularity and linear fractional
transforms}
\author{{\bf A. Lastra, S. Malek}\\
University of Alcal\'{a}, Departamento de F\'{i}sica y Matem\'{a}ticas,\\
Ap. de Correos 20, E-28871 Alcal\'{a} de Henares (Madrid), Spain,\\
University of Lille, Laboratoire Paul Painlev\'e,\\
59655 Villeneuve d'Ascq cedex, France,\\
{\tt alberto.lastra@uah.es}\\
{\tt Stephane.Malek@math.univ-lille1.fr }}
\date{July, 16 2018}
\maketitle
\thispagestyle{empty}
{ \small \begin{center}
{\bf Abstract}
\end{center}
This paper is a continuation of the work \cite{lama1} where parametric Gevrey asymptotics for singularly perturbed nonlinear PDEs has been studied.
Here, the partial differential operators are combined with particular Moebius transforms in the time variable. As a result, the leading term of the main
problem needs to be regularized by means of a singularly perturbed infinite order formal irregular operator that allows us to construct
a set of genuine solutions in the form of a Laplace transform in time and inverse Fourier transform in space. Furthermore, we obtain Gevrey
asymptotic expansions for these solutions of some order $K>1$ in the perturbation parameter.
\medskip

}
\bigskip \bigskip

\section{Introduction}
Within this paper, we focus on a family of nonlinear singularly perturbed equations which combines linear fractional transforms, partial derivatives
and differential operators of infinite order of the form
\begin{multline}
Q(\partial_{z})u(t,z,\epsilon) = \exp( \alpha \epsilon^{k} t^{k+1}\partial_{t} )R(\partial_{z})u(t,z,\epsilon)
+ P(t,\epsilon,\{ m_{\kappa,t,\epsilon} \}_{\kappa \in I},\partial_{t},\partial_{z})u(t,z,\epsilon)\\
+ Q_{1}(\partial_{z})u(t,z,\epsilon)Q_{2}(\partial_{z})u(t,z,\epsilon) + f(t,z,\epsilon) \label{main_eq_intro}
\end{multline}
where $\alpha,k>0$ are real numbers, $Q(X),R(X),Q_{1}(X),Q_{2}(X)$ stand for polynomials with complex coefficients and
$P(t,\epsilon,\{ U_{\kappa} \}_{\kappa \in I},V_{1},V_{2})$ represents a polynomial in $t,V_{1},V_{2}$, linear in $U_{\kappa}$, with
holomorphic coefficients w.r.t $\epsilon$ near the origin in $\mathbb{C}$, where the symbol $m_{\kappa,t,\epsilon}$ denotes a Moebius operator
acting on the time variable through
$$ m_{\kappa,t,\epsilon}u(t,z,\epsilon) = u( \frac{t}{1 + \kappa \epsilon t},z,\epsilon) $$
for $\kappa$ belonging to some finite subset $I$ of the positive real numbers $\mathbb{R}_{+}^{\ast}$. The forcing term $f(t,z,\epsilon)$ embodies
an analytic function in the vicinity of the origin relatively to $(t,\epsilon)$ and holomorphic w.r.t $z$ on a horizontal strip in $\mathbb{C}$
of the form $H_{\beta} = \{ z \in \mathbb{C} / |\mathrm{Im}(z)| < \beta \}$ for some $\beta>0$.

This work is a continuation of our previous study \cite{lama1} where we aimed attention at the next problem
\begin{equation}
Q(\partial_{z})\partial_{t}y(t,z,\epsilon) = H(t,\epsilon,\partial_{t},\partial_{z})y(t,z,\epsilon) +
Q_{1}(\partial_{z})y(t,z,\epsilon)Q_{2}(\partial_{z})y(t,z,\epsilon) + f(t,z,\epsilon) \label{lama1_equation_intro}
\end{equation}
for given vanishing initial data $y(0,z,\epsilon) \equiv 0$, where $Q_{1},Q_{2},H$ stand for polynomials and $f(t,z,\epsilon)$ is built up as above.
Under suitable constraints on the components of (\ref{lama1_equation_intro}), by means of Laplace and inverse Fourier transforms, we constructed
a set of genuine bounded holomorphic solutions $y_{p}(t,z,\epsilon)$, $0 \leq p \leq \varsigma - 1$, for some integer $\varsigma \geq 2$, defined
on domains $\mathcal{T} \times H_{\beta} \times \mathcal{E}_{p}$, for some well selected bounded sector $\mathcal{T}$ with vertex at 0
and $\underline{\mathcal{E}} = \{ \mathcal{E}_{p} \}_{0 \leq p \leq \varsigma - 1}$ a set of bounded sectors whose union contains a full
neighborhood of 0 in $\mathbb{C}^{\ast}$. On the sectors $\mathcal{E}_{p}$, the solutions $y_{p}$ are shown to share w.r.t $\epsilon$ a common
asymptotic expansion $\hat{y}(t,z,\epsilon) = \sum_{n \geq 0} y_{n}(t,z) \epsilon^{n}$ that represents a formal power series with bounded
holomorphic coefficients $y_{n}(t,z)$ on $\mathcal{T} \times H_{\beta}$. Furthermore, this asymptotic expansion turns out to be (at most) of Gevrey
order $1/k'$, for some integer $k' \geq 1$ (see Definition 7 for an explanation of this terminology) that comes out in the highest order term
of the differential operator $H$ which is of irregular type in the sense of \cite{man2} and displayed as
\begin{equation}
L(t,\epsilon,\partial_{t},\partial_{z}) = \epsilon^{(\delta_{D}-1)k'} t^{(\delta_{D}-1)(k'+1)} \partial_{t}^{\delta_{D}} R_{D}(\partial_{z})
\label{L_term_lama1_intro}
\end{equation}
for some integer $\delta_{D} \geq 2$ and a polynomial $R_{D}(X)$. In the case when the aperture of $\mathcal{E}_{p}$ can be taken slightly larger
than $\pi/k'$, the function $\epsilon \mapsto y_{p}(t,z,\epsilon)$ represents the $k'-$sum of $\hat{y}$ on $\mathcal{E}_{p}$ as described in
Definition 7.

Through the present contribution, our purpose is to carry out a comparable statement namely the existence of sectorial holomorphic solutions
and associated asymptotic expansions as $\epsilon$ tends to 0 with controlled Gevrey bounds. However, the appearance of the nonlocal Moebius
operator $m_{\kappa,t,\epsilon}$ changes drastically the whole picture in comparison with our previous investigation \cite{lama1}. Namely,
according to our approach, a leading term of finite order $\delta_{D} \geq 2$ in time as above (\ref{L_term_lama1_intro}) is insufficient to
ensure the construction of actual holomorphic solutions to our initial problem (\ref{main_eq_intro}). We need to supplant it by an
exponential formal differential operator
$$ \exp( \alpha \epsilon^{k} t^{k+1}\partial_{t} )R(\partial_{z}) = \sum_{p \geq 0} \frac{ (\alpha \epsilon^{k})^{p} }{p!}
(t^{k+1}\partial_{t})^{(p)}R(\partial_{z})
$$
of infinite order w.r.t $t$, where $(t^{k+1}\partial_{t})^{(p)}$ represents the $p-$th iterate of the irregular differential operator $t^{k+1}\partial{t}$.
As a result, (\ref{main_eq_intro}) becomes singularly perturbed of irregular type but of infinite order in time.
The reason for the choice of such a new leading term will be put into light later on in the introduction.

A similar regularization procedure has been introduced in a different context in the paper \cite{bafrparati} in order to obtain entire solutions in
space for hydrodynamical PDEs such as the 3D Navier Stokes equations
$$
 \partial_{t}v(t,x) + v(t,x) \cdot \nabla v(t,x) = - \nabla p(t,x) - \mu \Delta v(t,x) \ \ , \ \ \nabla \cdot v(t,x) = 0
$$
for given $2\pi-$periodic initial data $v(0,x) = v_{0}(x_{1},x_{2},x_{3})$ on $\mathbb{R}^{3}$, where the usual Laplacian
$\Delta = \sum_{j=1}^{3} \partial_{x_{j}}^{2}$ is asked to be replaced by a (pseudo differential) operator $\exp( \lambda A^{1/2})$,
where $\lambda>0$ and $A$ stands for the differential operator $-\nabla^{2}$, whose Fourier symbol is $\exp( \lambda |k|)$ for
$k \in \mathbb{Z}^{3} \setminus \{ 0 \}$. The resulting problem is shown to possess a solution $v(t,x)$ that is analytic w.r.t $x$ in
$\mathbb{C}^{3}$ for all $t > 0$ whereas the solutions of the initial problem are expected to develop singularities in space.

Under appropriate restrictions on the shape of (\ref{main_eq_intro}) listed in the statement of Theorem 1, we can select
\begin{enumerate}
\item a set $\underline{\mathcal{E}}$ of bounded sectors as mentioned above, which forms a so-called good covering in $\mathbb{C}^{\ast}$ (see Definition 5),
\item a bounded sector $\mathcal{T}$ with bisecting direction $d=0$
\item and a set of directions $\mathfrak{d}_{p} \in (-\frac{\pi}{2},\frac{\pi}{2})$,
$0 \leq p \leq \varsigma - 1$ organized in a way that the halflines $L_{\mathfrak{d}_{p}} = \mathbb{R}_{+}\exp( \sqrt{-1} \mathfrak{d}_{p} )$
avoid the infinite set of zeros of the map $\tau \mapsto Q(im) - \exp(\alpha k \tau^{k})R(im)$ for all $m \in \mathbb{R}$,
\end{enumerate}
for which we can exhibit a family of bounded holomorphic solutions $u_{p}(t,z,\epsilon)$ on the products $\mathcal{T} \times
H_{\beta} \times \mathcal{E}_{p}$. Each solution $u_{p}$ can be expressed as a Laplace transform of some order $k$ and Fourier inverse transform
\begin{equation}
u_{p}(t,z,\epsilon) = \frac{k}{(2\pi)^{1/2}} \int_{-\infty}^{+\infty} \int_{L_{\mathfrak{d}_{p}}}
 w^{\mathfrak{d}_{p}}(u,m,\epsilon) \exp( -(\frac{u}{\epsilon t})^{k} ) e^{izm} \frac{du}{u} dm \label{defin_up_intro}
\end{equation}
where $w^{\mathfrak{d}_{p}}(u,m,\epsilon)$ stands for a function with (at most) exponential growth of order $k$ on a sector
containing $L_{\mathfrak{d}_{p}}$ w.r.t $u$, owning exponential decay w.r.t $m$ on $\mathbb{R}$ and relying analytically on $\epsilon$
near 0 (Theorem 1). Moreover, we show that the functions $\epsilon \mapsto u_{p}(t,z,\epsilon)$ admit a common asymptotic expansion
$\hat{u}(t,z,\epsilon) = \sum_{m \geq 0} h_{m}(t,z) \epsilon^{m}$ on $\mathcal{E}_{p}$ that defines a formal power series with bounded
holomorphic coefficients on $\mathcal{T} \times H_{\beta}$. Besides, it turns out that this asymptotic expansion is (at most) of Gevrey order $1/k$
and leads to $k-$summability on $\mathcal{E}_{p_0}$ provided that one sector $\mathcal{E}_{p_0}$ has opening larger than $\pi/k$ (Theorem 2).

Another substantial contrast between the problems (\ref{main_eq_intro}) and (\ref{lama1_equation_intro}) lies in the fact that the real number $k$ is asked to
be less than 1. The situation $k=1$ is not covered by the technics developped in this work and is postponed for future inspection. However, the special
case $k=1$ has already been explored by the authors for some families of Cauchy problems and gives rise to double scale structures
involving 1 and so-called $1^{+}$ Gevrey estimates, see \cite{lama3}, \cite{ma2}. Observe that if one performs the change of variable $t=1/s$ through
the change of function $u(t,z,\epsilon)=X(s,z,\epsilon)$ then the equation (\ref{main_eq_intro}) is mapped into a singularly perturbed PDE combined
with small shifts $T_{\kappa,\epsilon}X(s,z,\epsilon) = X(s + \kappa \epsilon,z,\epsilon)$, for $\kappa \in I$. This restriction concerning the Gevrey order
of formal expansions of the analytic solutions is rather natural in the context of difference equations as observed by B. Braaksma and B. Faber in
\cite{brfa}. Namely, if $A(x)$ stands for an invertible matrix of dimension $n \geq 1$ with meromorphic coefficients at $\infty$ and $G(x,y)$ represents
a holomorphic function in $1/x$ and $y$ near $(\infty,0)$, under suitable assumptions on the formal fundamental matrix $\hat{Y}(x)$ of the linear equation
$y(x+1) = A(x)y(x)$, any formal solution $\hat{y}(x) \in \mathbb{C}^{n}[[1/x]]$ of the nonlinear difference equation
$$ y(x+1) -A(x)y(x) = G(x,y(x)) $$
can be decomposed as a sum of formal series $\hat{y}(x) = \sum_{h=1}^{q} \hat{y}_{h}(x)$ where each $\hat{y}_{h}(x)$ turns out to be
$k_{h}-$summable on suitable sectors for some real numbers $0 < k_{h} \leq 1$, for $1 \leq h \leq q$.

In order to construct the family of solutions $\{ u_{p} \}_{0 \leq p \leq \varsigma - 1}$ mentioned above, we follow an approach that has been
successfully applied by B. Faber and M. van der Put, see \cite{favan}, in the study of formal aspects of differential-difference operators such as the
construction of Newton polygons, factorizations and the extraction of formal solutions and consists in considering the translation
$x \mapsto x + \kappa$ as a formal differential operator of infinite order through the Taylor expansion at $x$, see (\ref{Taylor_at_x}). In our framework,
the action of the Moebius transform $T \mapsto \frac{T}{1 + \kappa T}$ is seen as an irregular operator of infinite order that can be formally written in
the exponential form
$$ \exp( - \kappa T^{2}\partial_{T} ) = \sum_{p \geq 0} (-1)^{p} \frac{\kappa^{p}}{p!}(T^{2}\partial_{T})^{(p)}. $$
If one seeks for genuine solutions in the form (\ref{defin_up_intro}), then the so-called \emph{Borel map}
$w^{\mathfrak{d}_{p}}(\tau,m,\epsilon)$ is asked to solve a related convolution equation (\ref{main_integral_eq_w}) that involves infinite order
operators $\exp( -\kappa \mathcal{C}_{k}(\tau) )$ where $\mathcal{C}_{k}(\tau)$ denotes the convolution map given by
(\ref{defin_mathcalCk}). It turns out that this operator $\exp( -\kappa \mathcal{C}_{k}(\tau) )$ acts on spaces of analytic functions
$f(\tau)$ with (at most) exponential growth of order $k$, i.e bounded by $C \exp( \nu |\tau|^{k})$ for some $C,\nu>0$ but increases strictly
the type $\nu$ by a quantity depending on $\kappa,k$ and $\nu$ as shown in Proposition 2, (\ref{bds_exp_kappaCk_Sd}). It is worthwhile mentioning that the
use of precise bounds for the so-called Wiman special function $E_{\alpha,\beta}(z) = \sum_{n \geq 0} z^{n}/\Gamma(\beta + \alpha n)$ for
$\alpha,\beta>0$ at infinity is crucial in the proof that the order $k$ is preserved under the action of
$\exp( -\kappa \mathcal{C}_{k}(\tau) )$. Notice that this function also played a central role in proving multisummability properties of formal solutions
in a perturbation parameter to certain families of nonlinear PDEs as described in our previous work \cite{lama2}. As a result, the presence of an
exponential type term $\exp( \alpha k \tau^{k} )$ in front of the equation (\ref{main_integral_eq_w}) and therefore the infinite order
operator $\exp( \alpha \epsilon^{k} t^{k+1} \partial_{t})$ as leading term of (\ref{main_eq_intro}) seems unavoidable to compensate such an exponential
growth.

We mention that a similar strategy has been carried out by S. $\mathrm{\bar{O}}$uchi in \cite{ou} who considered functional equations
$$ u(z) + \sum_{j=2}^{m} a_{j}u(z + z^{p}\varphi_{j}(z)) = f(z) $$
where $p \geq 1$ is an integer, $a_{j} \in \mathbb{C}^{\ast}$ and $\varphi_{j}(z)$,$f(z)$ stand for holomorphic functions near $z=0$. He established the
existence of formal power series solutions $\hat{u}(z) \in \mathbb{C}[[z]]$ that are proved to be $p-$summable in suitable directions by solving an
associated convolution equation of infinite order for the Borel transform of order $p$ in analytic functions spaces with (at most) exponential
growth of order $p$ on convenient unbounded sectors. More recently, in a work in progress \cite{hitaya}, S. Hirose, H. Yamazawa and H. Tahara
are extending the above
statement to more general functional PDEs such as
$$ u(t,x) = a_{1}(t,x)(t\partial_{t})^{2}(u(t+t^{2},x)) + a_{2}(t,x)\partial_{x}u(t+t^{2},x) + f(t,x) $$
for analytic coefficients $a_{1},a_{2},f$ near $0 \in \mathbb{C}^{2}$ for which formal series solutions
$$\hat{u}(t,x) = \sum_{n \geq 1} u_{n}(x)t^{n}$$
can be built up that are shown to be multisummable in appropriate multidirections in the sense defined in \cite{ba}.

In a wider framework, there exists a gigantic literature dealing with infinite order PDEs/ODEs both in
mathematics and in theoretical physics. We just quote some recent references somehow related to
our research interests. In the paper \cite{aokakota}, the authors study formal solutions and their Borel transform of singularly perturbed
differential equations of infinite order
$$ \sum_{j \geq 0} \epsilon^{j} P_{j}(x,\epsilon \partial_{x}) \psi(x,\epsilon) = 0 $$
where $P_{j}(x,\xi) = \sum_{k \geq 0} a_{j,k}(x) \xi^{k}$ represent entire functions with appropriate growth features. For a nice introduction of the point of
view introduced by M. Sato called algebraic microlocal analysis, we refer to \cite{kast}. Other important contributions on infinite order ODEs in this
context of algebraic microlocal analysis can be singled out such as \cite{kawst1}, \cite{kawst2}.

\noindent The paper is arranged as follows.\\
In Section 2, we remind the reader the definition of Laplace transform for an order $k$ chosen among the positive real numbers and basic
formulas for the Fourier inverse transform acting on exponentially flat functions.\\
In Section 3, we display our main problem (\ref{main_ivp_u}) and describe the strategy used to solve it. In a first step, we restrain our inquiry for the
sets of solutions to time rescaled function spaces, see (\ref{time_rescaled_solution}). Then, we elect as potential candidates for solutions Laplace
transforms of order $k$ and Fourier inverse transforms of \emph{Borel maps} $w$ with exponential growth on unbounded sectors and exponential decay on
the real line. In the last step, we write down the convolution problem (\ref{main_integral_eq_w}) which is asked to be solved by the map $w$.\\
In Section 4, we analyze bounds for linear/nonlinear convolution operators of finite/infinite orders acting on different spaces of analytic functions
on sectors.\\
In Section 5, we solve the principal convolution problem (\ref{main_integral_eq_w}) within the Banach spaces described in Sections 3 and 4 by means of
a fixed point argument.\\
In Section 6, we provide a set of genuine holomorphic solutions (\ref{Laplace_Fourier_up}) to our initial equation (\ref{main_ivp_u}) by
executing backwards the lines of argument described in Section 3. Furthermore, we show that the difference of any two neighboring solutions
tends to 0, for $\epsilon$ in the vicinity of the origin, faster than a function with exponential decay of order $k$.\\
Finally, in Section 7, we prove the existence of a common asymptotic expansion of Gevrey order $1/k>1$ for the solutions mentioned above leaning
on the flatness estimates reached in Section 6, by means of a theorem by Ramis and Sibuya. 

\section{Laplace, Borel transforms of order $k$ and Fourier inverse maps}

We recall the definition of Laplace transform of order $k$ as introduced in \cite{lama1} but here the order $k$ is assumed to be a real
number less than 1 and larger than 1/2. If $z \in \mathbb{C}^{\ast}$ denotes a non vanishing complex number, we set
$z^{k} = \exp( k \log(z) )$ where $\log(z)$ stands for the principal value of the complex logarithm defined as $\log(z) = \log|z| + i \mathrm{arg}(z)$ with
$-\pi < \mathrm{arg}(z) < \pi$.
\begin{defin} Let $\frac{1}{2} < k < 1$ be a real number. Let
$S_{d,\delta} = \{ \tau \in \mathbb{C}^{\ast} : |d - \mathrm{arg}(\tau)| < \delta \}$ be some unbounded sector with bisecting direction
$d \in \mathbb{R}$ and aperture $2\delta > 0$. 

Consider a holomorphic function $w : S_{d,\delta} \rightarrow \mathbb{C}$ that withstands the bounds :
there exist $C>0$ and $K>0$ such that
$$ |w(\tau)| \leq C |\tau|^{k} \exp( K |\tau|^{k} ) $$
for all $\tau \in S_{d,\delta}$. We define the Laplace transform of $w$ of order $k$ in the direction $d$ as the integral transform
$$ \mathcal{L}_{k}^{d}(w)(T) = k \int_{L_{\gamma}} w(u) \exp( -(\frac{u}{T})^{k} ) \frac{du}{u} $$
along a half-line $L_{\gamma} = \mathbb{R}_{+}e^{\sqrt{-1}\gamma} \subset S_{d,\delta} \cup \{ 0 \}$, where $\gamma$ depends on
$T$ and is chosen in such a way that $\cos(k(\gamma - \mathrm{arg}(T))) \geq \delta_{1} > 0$, for some fixed $\delta_{1}$.
The function $\mathcal{L}^{d}_{k}(w)(T)$ is well defined, holomorphic and bounded on any sector
$$ S_{d,\theta,R^{1/k}} = \{ T \in \mathbb{C}^{\ast} : |T| < R^{1/k} \ \ , \ \ |d - \mathrm{arg}(T) | < \theta/2 \},$$
where $\frac{\pi}{k} < \theta < \frac{\pi}{k} + 2\delta$ and
$0 < R < \delta_{1}/K$.
\end{defin}

We restate the definition of some family of Banach spaces introduced in \cite{lama1}.
\begin{defin} Let $\beta, \mu \in \mathbb{R}$. We set
$E_{(\beta,\mu)}$ as the vector space of continuous functions $h : \mathbb{R} \rightarrow \mathbb{C}$ such that
$$ ||h(m)||_{(\beta,\mu)} = \sup_{m \in \mathbb{R}} (1+|m|)^{\mu} \exp( \beta |m|) |h(m)| $$
is finite. The space $E_{(\beta,\mu)}$ endowed with the norm $||.||_{(\beta,\mu)}$ becomes a Banach space.
\end{defin}

Finally, we remind the reader the definition of the inverse Fourier transform acting on the latter Banach spaces and some of its
handy formulas relative to derivation and convolution product as stated in
\cite{lama1}.

\begin{defin}
Let $f \in E_{(\beta,\mu)}$ with $\beta > 0$, $\mu > 1$. The inverse Fourier transform of $f$ is given by
$$ \mathcal{F}^{-1}(f)(x) = \frac{1}{ (2\pi)^{1/2} } \int_{-\infty}^{+\infty} f(m) \exp( ixm ) dm $$
for all $x \in \mathbb{R}$. The function $\mathcal{F}^{-1}(f)$ extends to an analytic bounded function on the strips
\begin{equation}
H_{\beta'} = \{ z \in \mathbb{C} / |\mathrm{Im}(z)| < \beta' \}. \label{strip_H_beta}
\end{equation}
for all given $0 < \beta' < \beta$.\\
a) Define the function $m \mapsto \phi(m) = imf(m)$ which belongs to the space $E_{(\beta,\mu-1)}$. Then, the next identity
$$ \partial_{z} \mathcal{F}^{-1}(f)(z) = \mathcal{F}^{-1}(\phi)(z) $$
occurs.\\
b) Take $g \in E_{(\beta,\mu)}$ and set 
$$ \psi(m) = \frac{1}{(2\pi)^{1/2}} \int_{-\infty}^{+\infty} f(m-m_{1})g(m_{1}) dm_{1} $$
as the convolution product of $f$ and $g$. Then, $\psi$ belongs to $E_{(\beta,\mu)}$ and moreover,
$$ \mathcal{F}^{-1}(f)(z) \mathcal{F}^{-1}(g)(z) = \mathcal{F}^{-1}(\psi)(z) $$
for all $z \in H_{\beta}$.
\end{defin}

\section{Outline of the main initial value problem and related auxiliary problems}

We set $k \in (\frac{1}{2},1)$ as a real number. Let $D \geq 2$ be an integer, $\alpha_{D} >0$ be a positive real number
and $c_{12},c_{f}$ be complex numbers in $\mathbb{C}^{\ast}$. For
$1 \leq l \leq D-1$, we consider complex numbers $c_{l} \in \mathbb{C}^{\ast}$ and non negative integers $d_{l},\delta_{l},\Delta_{l}$,
together with positive real numbers
$\kappa_{l}>0$ submitted to the next constraints. We assume that
\begin{equation}
1 = \delta_{1} \ \ , \ \ \delta_{l} < \delta_{l+1} \label{defin_deltal}
\end{equation}
for all $1 \leq l \leq D-2$. We also take for granted that
\begin{equation}
d_{l} > \delta_{l}(k+1) \ \ , \ \ \Delta_{l}-d_{l}+\delta_{l} \geq 0 \ \ , \ \ \label{cond_dl_deltal_Deltal}
\end{equation}
whenever $1 \leq l \leq D-1$. Let $Q(X),Q_{1}(X),Q_{2}(X),R_{l}(X) \in \mathbb{C}[X]$, $1 \leq l \leq D$, be polynomials such that
\begin{multline}
\mathrm{deg}(Q) = \mathrm{deg}(R_{D}) \geq \mathrm{deg}(R_{l}) \ \ , \ \ \mathrm{deg}(R_{D}) \geq \mathrm{deg}(Q_{1}) \ \ , \ \
\mathrm{deg}(R_{D}) \geq \mathrm{deg}(Q_{2}), \\
Q(im) \neq 0 \ \ , \ \ R_{D}(im) \neq 0 \label{constraints_degree_coeff}
\end{multline}
for all $m \in \mathbb{R}$, all $1 \leq l \leq D-1$. 

We consider a sequence of functions
$m \mapsto F_{n}(m,\epsilon)$, for $n \geq 1$ that belong to the Banach space $E_{(\beta,\mu)}$ for some $\beta > 0$ and
$\mu > \max( \mathrm{deg}(Q_{1})+1, \mathrm{deg}(Q_{2})+1)$ and that depend analytically on $\epsilon \in D(0,\epsilon_{0})$, where
$D(0,\epsilon_{0})$ denotes the open disc centered at 0 in $\mathbb{C}$ with radius $\epsilon_{0}>0$. We assume that there
exist constants $K_{0},T_{0}>0$ such that
\begin{equation}
||F_{n}(m,\epsilon)||_{(\beta,\mu)} \leq K_{0} (\frac{1}{T_{0}})^{n} \label{norm_beta_mu_F_n_a_ln}
\end{equation}
for all $n \geq 1$, for all $\epsilon \in D(0,\epsilon_{0})$. We define
$$ F(T,z,\epsilon) = \sum_{n \geq 1} \mathcal{F}^{-1}( m \mapsto F_{n}(m,\epsilon))(z) T^{n}$$
which represents a convergent series on $D(0,T_{0}/2)$ with holomorphic and bounded coefficients on
$H_{\beta'}$ for any given width $0 < \beta' < \beta$. For all $1 \leq l \leq D-1$, we set the polynomials
$A_{l}(T,\epsilon) = \sum_{n \in I_{l}} A_{l,n}(\epsilon)T^{n}$ where $I_{l}$ are finite subsets of $\mathbb{N}$ and
$A_{l,n}(\epsilon)$ represent bounded holomorphic functions on the disc $D(0,\epsilon_{0})$. We put
$$ f(t,z,\epsilon) = F(\epsilon t, z, \epsilon) \ \ , \ \
a_{l}(t,\epsilon) = A_{l}(\epsilon t, \epsilon) $$
for all $1 \leq l \leq D-1$. By construction, $f(t,z,\epsilon)$ (resp. $a_{l}(t,\epsilon)$)
defines a bounded holomorphic function on $D(0,r) \times H_{\beta'} \times D(0,\epsilon_{0})$ (resp. $D(0,r) \times D(0,\epsilon_{0})$)
for any given $0 < \beta' < \beta$ and radii $r,\epsilon_{0}>0$ with $r \epsilon_{0} \leq T_{0}/2$.

Let us introduce the next differential operator of infinite order formally defined as
\begin{equation}
\exp( \alpha_{D} \epsilon^{k} t^{k+1} \partial_{t} ) = \sum_{p \geq 0} \frac{(\alpha_{D} \epsilon^{k})^{p}}{p!} (t^{k+1}\partial_{t})^{(p)}
\end{equation}
where $(t^{k+1}\partial_{t})^{(p)}$ stands for the $p-$th iterate of the differential operator $t^{k+1}\partial_{t}$.
We consider a family of nonlinear singularly perturbed initial value problems which involves this latter operator of infinite order as leading term and
linear fractional transforms
\begin{multline}
Q(\partial_{z}) u(t,z,\epsilon) = \exp( \alpha_{D} \epsilon^{k} t^{k+1} \partial_{t} )
R_{D}(\partial_{z})u(t,z,\epsilon) \\
+ \sum_{l=1}^{D-1} \epsilon^{\Delta_{l}}c_{l}a_{l}(t,\epsilon)t^{d_{l}}R_{l}(\partial_{z})
\partial_{t}^{\delta_l}( u( \frac{t}{1 + \kappa_{l} \epsilon t},z,\epsilon ) )
+ c_{12}Q_{1}(\partial_{z})u(t,z,\epsilon)Q_{2}(\partial_{z})u(t,z,\epsilon) + c_{f}f(t,z,\epsilon) \label{main_ivp_u}
\end{multline}
for vanishing initial data $u(0,z,\epsilon) = 0$.

Within this work, we search for time rescaled solutions of (\ref{main_ivp_u}) of the form
\begin{equation}
u(t,z,\epsilon) = U(\epsilon t , z ,\epsilon) \label{time_rescaled_solution}
\end{equation}
Then, through the change of variable $T=\epsilon t$, the expression $U(T,z,\epsilon)$ is subjected to solve the next nonlinear singular
problem involving fractional transforms
\begin{multline}
Q(\partial_{z})U(T,z,\epsilon) = \exp( \alpha_{D} T^{k+1} \partial_{T} )
R_{D}(\partial_{z}) U(T,z,\epsilon) \\
+ \sum_{l=1}^{D-1} \epsilon^{\Delta_{l} - d_{l} + \delta_{l}}c_{l}A_{l}(T,\epsilon) T^{d_l} R_{l}(\partial_{z}) \partial_{T}^{\delta_l} \left(
U(\frac{T}{1 + \kappa_{l}T},z,\epsilon) \right) \\
+ c_{12}Q_{1}(\partial_{z})U(T,z,\epsilon)Q_{2}(\partial_{z})U(T,z,\epsilon) + c_{f}F(T,z,\epsilon) \label{main_ivp_U}
\end{multline}
for given initial data $U(0,z,\epsilon) = 0$. According to the assumption (\ref{cond_dl_deltal_Deltal}), there exists a real
number $d_{l,k}>0$ with
\begin{equation}
d_{l} = \delta_{l}(k+1) + d_{l,k} \label{defin_dlk}
\end{equation}
for all $1 \leq l \leq D-1$. Besides, with the help of the formula (8.7) from \cite{taya} p. 3630, we can expand the next differential operators
\begin{equation}
T^{\delta_{l}(k+1)} \partial_{T}^{\delta_l} = (T^{k+1}\partial_{T})^{\delta_l} +
\sum_{1 \leq p \leq \delta_{l}-1} A_{\delta_{l},p}T^{k(\delta_{l} -p)}(T^{k+1}\partial_{T})^{p} \label{Tahara_expand}
\end{equation}
where $A_{\delta_{l},p} \in \mathbb{R}$, for $1 \leq p \leq \delta_{l}-1$ and $1 \leq l \leq D-1$. Hence, according to (\ref{defin_dlk}) together
with (\ref{Tahara_expand}), we can write down the next equation for $U(T,z,\epsilon)$, namely
\begin{multline}
Q(\partial_{z})U(T,z,\epsilon)
= \exp( \alpha_{D}T^{k+1}\partial_{T}) R_{D}(\partial_{z})U(T,z,\epsilon)
+ \sum_{l=1}^{D-1} \epsilon^{\Delta_{l}-d_{l}+\delta_{l}} R_{l}(\partial_{z}) c_{l}A_{l}(T,\epsilon)\\
\times T^{d_{l,k}} \left( (T^{k+1}\partial_{T})^{\delta_{l}} +
\sum_{1 \leq p \leq \delta_{l}-1} A_{\delta_{l},p} T^{k(\delta_{l}-p)}(T^{k+1}\partial_{T})^{p} \right) \left(
U(\frac{T}{1 + \kappa_{l}T},z,\epsilon) \right)\\
+ c_{12}Q_{1}(\partial_{z})U(T,z,\epsilon)Q_{2}(\partial_{z})U(T,z,\epsilon) + c_{f}F(T,z,\epsilon).
\label{main_ivp_U_prep_form}
\end{multline}

We now provide the definition of a modified version of some Banach spaces introduced in the papers \cite{lama1}, \cite{lama2} that takes into account
a ramified variable $\tau^{k}$ for $k$ given as above.

\begin{defin} Let $S_{d}$ be an unbounded sector centered at 0
with bisecting direction $d \in \mathbb{R}$. Let $\nu,\beta,\mu>0$ and $\rho>0$ be positive real numbers. Let $k \in (\frac{1}{2},1)$ defined
as above. We set
$F_{(\nu,\beta,\mu,k,\rho)}^{d}$ as the vector space of continuous functions $(\tau,m) \mapsto h(\tau,m)$ on
$S_{d} \times \mathbb{R}$, which are holomorphic with respect to $\tau$ on $S_{d}$ such that\\
1) For all $m \in \mathbb{R}$, the function $\tau \mapsto h(\tau,m)$ extends analytically on some cut disc $D(0,\rho) \setminus L_{-}$ where
$L_{-}$ denotes the segment $(-\rho,0]$.\\
2) The norm
$$ ||h(\tau,m)||_{(\nu,\beta,\mu,k,\rho)} = \sup_{\tau \in S_{d} \cup D(0,\rho) \setminus L_{-}} (1+|m|)^{\mu} \frac{1 + |\tau|^{2k}}{|\tau|^{k}}
e^{\beta |m| - \nu |\tau|^{k}}|h(\tau,m)| $$
is finite.

The space
$F_{(\nu,\beta,\mu,k,\rho)}^{d}$ equipped with the norm
$||.||_{(\nu,\beta,\mu,k,\rho)}$ forms a Banach space.
\end{defin}

\begin{lemma} For $\beta,\mu$ given in (\ref{norm_beta_mu_F_n_a_ln}), there exists $\nu>0$ such that the series
$$ \psi(\tau,m,\epsilon) = \sum_{n \geq 1} F_{n}(m,\epsilon) \frac{\tau^n}{\Gamma( \frac{n}{k})} $$
define a function that belongs to the space
$F_{(\nu,\beta,\mu,k,\rho)}^{d}$ for all $\epsilon \in D(0,\epsilon_{0})$, for any
radius $\rho>0$, any sector $S_{d}$ for $d \in \mathbb{R}$.
\end{lemma}

\begin{proof} By Definition of the norm $||.||_{(\nu,\beta,\mu,k,\rho)}$, we get the next upper bounds
\begin{equation}
 ||\psi(\tau,m,\epsilon)||_{(\nu,\beta,\mu,k,\rho)} \leq \sum_{n \geq 1}
||F_{n}(m,\epsilon)||_{(\beta,\mu)} (\sup_{\tau \in (D(0,\rho)\setminus L_{-}) \cup S_{d}}
\frac{1 + |\tau|^{2k}}{|\tau|^{k}} \exp(-\nu |\tau|^{k})
\frac{|\tau|^n}{\Gamma(\frac{n}{k})}) \label{maj_norm_psi_1}
\end{equation}
Due to the classical estimates
$$ \sup_{x \geq 0} x^{m_1} \exp( -m_{2} x ) = (\frac{m_{1}}{m_{2}})^{m_1} e^{-m_{1}} $$
for any real numbers $m_{1} \geq 0$, $m_{2}>0$, together with the Stirling formula (see \cite{ba2}, Appendix B.3)
$$ \Gamma(n/k) \sim (2\pi)^{1/2}(n/k)^{\frac{n}{k}-\frac{1}{2}}e^{-n/k}$$
as $n$ tends to $+\infty$, we get two constants
$A_{1}>0$ depending on $k,\nu$ and $A_{2}>0$ depending on $k$ such that
\begin{multline}
\sup_{\tau \in (D(0,\rho) \setminus L_{-}) \cup S_{d}}
\frac{1 + |\tau|^{2k}}{|\tau|^{k}} \exp(-\nu |\tau|^{k})
\frac{|\tau|^n}{\Gamma(\frac{n}{k})} \leq \sup_{x \geq 0} (1+x^{2})x^{\frac{n}{k}-1}
\frac{e^{-\nu x}}{\Gamma(\frac{n}{k})} \\
\leq \left( (\frac{ \frac{n}{k} - 1}{\nu})^{\frac{n}{k}-1} e^{-( \frac{n}{k} - 1)} +
(\frac{ \frac{n}{k} + 1}{\nu} )^{\frac{n}{k} + 1} e^{-(\frac{n}{k} + 1)} \right)/ \Gamma(n/k) 
\leq A_{1}(\frac{A_{2}}{\nu^{1/k}})^{n} \label{sup_Stirling}
\end{multline}
for all $n \geq 1$. Therefore, if $\nu^{1/k} > A_{2}/T_{0}$ then we obtain the bounds
\begin{equation}
||\psi(\tau,m,\epsilon)||_{(\nu,\beta,\mu,k,\rho)} \leq A_{1}K_{0} \sum_{n \geq 1} (\frac{A_{2}}{T_{0}\nu^{1/k}})^{n}
= \frac{A_{1}K_{0}A_{2}}{T_{0} \nu^{1/k}} \frac{1}{1 - \frac{A_{2}}{T_{0}\nu^{1/k}}}
\end{equation}
for all $\epsilon \in D(0,\epsilon_{0})$.
\end{proof}
By construction, according to the very definition of the Gamma function, the function
$F(T,z,\epsilon)$ can be represented as a Laplace transform of order $k$ in direction $d$ and Fourier inverse transform
\begin{equation}
F(T,z,\epsilon) = \frac{k}{(2\pi)^{1/2}} \int_{-\infty}^{+\infty} \int_{L_{\gamma}} \psi(u,m,\epsilon)
\exp( -(\frac{u}{T})^{k} ) e^{izm} \frac{du}{u} dm 
\end{equation}
where the integration path $L_{\gamma} = \mathbb{R}_{+}e^{\sqrt{-1}\gamma}$ stands for a halfline with direction $\gamma \in \mathbb{R}$
which belongs to the set $S_{d} \cup \{ 0 \}$, whenever $T$ belongs to a sector $S_{d,\theta,\varrho}$ with
bisecting direction $d$, aperture $\frac{\pi}{k} < \theta < \frac{\pi}{k} + \mathrm{Ap}(S_{d})$ and radius $\varrho$
with $\mathrm{Ap}(S_{d})$ the aperture of $S_{d}$ for some $\varrho>0$ and $z$ appertains to a strip $H_{\beta'}$ for any
$0 < \beta' < \beta$ together with $\epsilon \in D(0,\epsilon_{0})$.

In the next step, we seek for solutions $U(T,z,\epsilon)$ of (\ref{main_ivp_U_prep_form}) on the same domains as above that can be
expressed similarly to $F(T,z,\epsilon)$ as integral
representations through Laplace transforms of order $k$ and Fourier inverse transform
\begin{equation}
U_{\gamma}(T,z,\epsilon) = \frac{k}{(2\pi)^{1/2}} \int_{-\infty}^{+\infty} \int_{L_{\gamma}} w(u,m,\epsilon)
\exp( -(\frac{u}{T})^{k} ) e^{izm} \frac{du}{u} dm
\end{equation}
Our goal is the statement of a related problem fulfilled by the expression $w(\tau,m,\epsilon)$ that is forecast
to be solved in the next section among the Banach spaces introduced above in Definition 4. Overall this section, let us assume that
the function $w(\tau,m,\epsilon)$ belongs to the Banach space
$F_{(\nu,\beta,\mu,k,\rho)}^{d}$.

We first display some formulas related to the action of the differential operators of irregular type and multiplication by
monomials. A similar statement has been given in Section 3 of \cite{lama1} for formal series expansions.

\begin{lemma} 1) The action of the differential operator $T^{k+1}\partial_{T}$ on $U_{\gamma}$ is given by
\begin{equation}
T^{k+1}\partial_{T}U_{\gamma}(T,z,\epsilon) =
\frac{k}{(2\pi)^{1/2}} \int_{-\infty}^{+\infty} \int_{L_{\gamma}} ku^{k}w(u,m,\epsilon)
\exp( - (\frac{u}{T})^{k} ) e^{izm} \frac{du}{u} dm. \label{TkpartialTUgamma}
\end{equation}
2) Let $m' > 0$ be a real number. The action of the multiplication by $T^{m'}$ on $U_{\gamma}$ is described through
\begin{multline}
T^{m'}U_{\gamma}(T,z,\epsilon) =
\frac{k}{(2\pi)^{1/2}} \int_{-\infty}^{+\infty} \int_{L_{\gamma}} \left(
\frac{u^{k}}{\Gamma(\frac{m'}{k})} \int_{0}^{u^{k}} (u^{k} - s)^{\frac{m'}{k}-1} w(s^{1/k},m,\epsilon) \frac{ds}{s} \right)\\
\times \exp( - (\frac{u}{T})^{k} ) e^{izm} \frac{du}{u} dm \label{TmUgamma}
\end{multline}
3) The action of the differential operators $Q(\partial_{z})$ and multiplication with the resulting functions $Q(\partial_{z})U_{\gamma}$
maps $U_{\gamma}$ into a Laplace and Fourier transform,
\begin{multline}
Q_{1}(\partial_{z})U_{\gamma}(T,z,\epsilon)Q_{2}(\partial_{z})U_{\gamma}(T,z,\epsilon) =
\frac{k}{(2\pi)^{1/2}} \int_{-\infty}^{+\infty} \int_{L_{\gamma}}
\left( \frac{1}{(2\pi)^{1/2}} \int_{-\infty}^{+\infty} u^{k} \right. \\
\left. \times \int_{0}^{u^k} Q_{1}(i(m-m_{1}))
w((u^{k}-s)^{1/k},m-m_{1},\epsilon) Q_{2}(im_{1})w(s^{1/k},m_{1},\epsilon) \frac{1}{(u^{k}-s)s} ds dm_{1} \right)\\
\times \exp( - (\frac{u}{T})^{k} ) e^{izm} \frac{du}{u} dm \label{QUgammaQUgamma}
\end{multline}
\end{lemma}
\begin{proof} Here we present direct analytic proofs which avoids the use of summability arguments through the Watson's lemma.
The first point 1) is obtained by a mere derivation under the $\int$ symbol. We turn to the second point 2). By application of the
Fubini theorem we get that
\begin{multline*}
A = \int_{L_{\gamma}} u^{k-1} \int_{0}^{u^k} (u^{k} - s)^{\frac{m'}{k}-1} w(s^{1/k},m,\epsilon) \frac{ds}{s}
\exp( - (\frac{u}{T})^{k} ) du \\
= \int_{L_{\gamma'}} \left( \int_{L_{s^{1/k},\gamma}} u^{k-1}(u^{k}-s)^{\frac{m'}{k}-1} \exp( - (\frac{u}{T})^{k} ) du \right)
w(s^{1/k},m,\epsilon) \frac{ds}{s}
\end{multline*}
where $\gamma' = k \gamma$ and $L_{s^{1/k},\gamma} = [|s|^{1/k},+\infty) e^{\sqrt{-1}\gamma}$. On the other hand, by successive path
deformations
$u^{k}=v$, and $v-s=v'$ we get that
$$
\int_{L_{s^{1/k},\gamma}} u^{k-1}(u^{k}-s)^{\frac{m'}{k}-1} \exp( - (\frac{u}{T})^{k} ) du =
\int_{L_{s,\gamma'}} (v-s)^{\frac{m'}{k}-1} \exp( -\frac{v}{T^{k}} ) \frac{1}{k} dv
$$
where $L_{s,\gamma'} = [|s|,+\infty) e^{\sqrt{-1}\gamma'}$ and
$$
\int_{L_{s,\gamma'}} (v-s)^{\frac{m'}{k}-1} \exp( -\frac{v}{T^{k}} ) \frac{1}{k} dv = 
\int_{L_{\gamma'}} (v')^{\frac{m'}{k}-1} \exp( -\frac{v'}{T^{k}} ) \frac{1}{k} dv' \exp( -\frac{s}{T^{k}} )
$$
By the very definition of the Gamma function and a path deformation yields
$$ \int_{L_{\gamma'}} (v')^{\frac{m'}{k}-1} \exp( -\frac{v'}{T^{k}} ) dv' = \Gamma(\frac{m'}{k})T^{m'} $$
As a result, according to the path deformation $s=u^{k}$, we finally get
$$ A = \int_{L_{\gamma'}} \frac{\Gamma(\frac{m'}{k})}{k} T^{m'} w(s^{1/k},m,\epsilon) \exp( -\frac{s}{T^{k}} ) \frac{ds}{s}
= \int_{L_{\gamma}} w(u,m,\epsilon) \exp( -(\frac{u}{T})^{k} ) \frac{du}{u} \Gamma(\frac{m'}{k}) T^{m'} $$
which implies the identity (\ref{TmUgamma}).

We aim our attention to the point 3). Again the Fubini theorem yields
\begin{multline*}
A = \int_{-\infty}^{+\infty} \int_{-\infty}^{+\infty} \int_{L_{\gamma}} \int_{0}^{u^k}
u^{k-1} Q_{1}(i(m-m_{1}))
w((u^{k} - s)^{1/k},m-m_{1},\epsilon)\\
\times Q_{2}(im_{1})w(s^{1/k},m_{1},\epsilon) \frac{1}{(u^{k}-s)s}
\exp( -(\frac{u}{T})^{k} ) e^{izm} ds du dm_{1} dm =\\
\int_{-\infty}^{+\infty} \int_{-\infty}^{+\infty} \int_{L_{\gamma'}} \int_{L_{s^{1/k},\gamma}}
u^{k-1} Q_{1}(i(m-m_{1}))
w((u^{k} - s)^{1/k},m-m_{1},\epsilon)\\
\times Q_{2}(im_{1})w(s^{1/k},m_{1},\epsilon) \frac{1}{(u^{k}-s)s}
\exp( -(\frac{u}{T})^{k} ) e^{izm} du ds dm dm_{1}
\end{multline*}
where $\gamma' = k \gamma$ and $L_{s^{1/k},\gamma} = [|s|^{1/k},+\infty)e^{\sqrt{-1}\gamma}$. By the path deformation $v=u^{k}$,
\begin{multline*}
B = \int_{L_{s^{1/k},\gamma}} u^{k-1}w((u^{k} - s)^{1/k},m-m_{1},\epsilon) \frac{1}{(u^{k}-s)} \exp( -(\frac{u}{T})^{k} ) du
\\
= \int_{L_{s,\gamma'}} \frac{1}{k} w((v-s)^{1/k},m-m_{1},\epsilon) \frac{1}{v-s} \exp( -\frac{v}{T^{k}} ) dv
\end{multline*}
in a row with the path deformation $v-s = v'$,
$$
B = \int_{L_{\gamma'}} \frac{1}{k} w((v')^{1/k},m-m_{1},\epsilon) \frac{1}{v'} \exp( -\frac{v'}{T^{k}} ) dv'
\exp( -\frac{s}{T^{k}} ). $$
Therefore, we obtain
\begin{multline*}
A = \int_{-\infty}^{+\infty} \int_{-\infty}^{+\infty} \int_{L_{\gamma'}} \int_{L_{\gamma'}}
Q_{1}(i(m-m_{1}))\frac{1}{k} w((v')^{1/k},m-m_{1},\epsilon) \\
\times \frac{1}{v'} \exp( -\frac{v'}{T^{k}} )
\exp( -\frac{s}{T^{k}} ) Q_{2}(im_{1})w(s^{1/k},m_{1},\epsilon) \frac{1}{s} e^{izm} dv' ds dm dm_{1}
\end{multline*}
Besides, by the change of variable $m-m_{1}=m'$,
\begin{multline*}
\int_{-\infty}^{+\infty} Q_{1}(i(m-m_{1}))w((v')^{1/k},m-m_{1},\epsilon) e^{izm} dm\\
= \int_{-\infty}^{+\infty} Q_{1}(im') w((v')^{1/k},m',\epsilon)e^{izm'} dm' e^{izm_{1}}
\end{multline*}
As a result,
\begin{multline*}
A = \frac{1}{k} \int_{-\infty}^{+\infty} \int_{L_{\gamma'}}
Q_{1}(im')w((v')^{1/k},m',\epsilon) \frac{1}{v'} \exp( -\frac{v'}{T^{k}} ) e^{izm'} dv' dm'\\
\times
\int_{-\infty}^{+\infty} \int_{L_{\gamma'}}
Q_{2}(im_{1})w(s^{1/k},m_{1},\epsilon) \frac{1}{s} \exp( -\frac{s}{T^{k}} ) e^{izm_{1}} ds dm_{1}
\end{multline*}
and according to the paths deformations $s=u^{k}$ and $v'=u^{k}$, we get at last
\begin{multline*}
A = k \int_{-\infty}^{+\infty} \int_{L_{\gamma}} Q_{1}(im')w(u,m',\epsilon) \exp( - (\frac{u}{T})^{k} )
e^{izm'} \frac{du}{u} dm'\\
\times \int_{-\infty}^{+\infty} \int_{L_{\gamma}} Q_{2}(im_{1}) w(u,m_{1},\epsilon)
\exp( - (\frac{u}{T})^{k} ) e^{izm_{1}} \frac{du}{u} dm_{1}
\end{multline*}
from which the identity (\ref{QUgammaQUgamma}) follows.
\end{proof}
At the next level, we describe the action of the Moebius transform $T \mapsto \frac{T}{1 + \kappa_{l}T}$ on $U_{\gamma}$.
It needs some preliminaries.

We depart as in the work of B. Faber and M. Van der Put \cite{favan} which describes the translation $x \mapsto x + \kappa_{l}$ as
a differential operator of infinite order through the Taylor expansion. Namely, for any holomorphic function
$f : U \mapsto \mathbb{C}$ defined on an open convex set $U \subset \mathbb{C}$ containing $x$ and $x + \kappa_{l}$, the next
Taylor formula holds
\begin{equation}
f(x + \kappa_{l}) = \sum_{p \geq 0} \frac{f^{(p)}(x)}{p!} \kappa_{l}^{p} \label{Taylor_at_x}
\end{equation}
where $f^{(p)}(x)$ denotes the derivative of order $p \geq 0$ of $f$ where by convention $f^{(0)}(x)=f(x)$. If one performs the change of
variable $x=1/T$ through the change of function $f(x) = U(1/x)$, one obtains a corresponding formula for $U(T)$,
\begin{equation}
U(\frac{T}{1 + \kappa_{l}T}) = \sum_{p \geq 0} \frac{(-1)^{p}(T^{2}\partial_{T})^{(p)}}{p!} \kappa_{l}^{p} U(T) \label{fraction_inf_order_irregular}
\end{equation}
where $(T^{2}\partial_{T})^{(p)}$ represents the $p-$th iterate of the irregular operator $T^{2}\partial_{T}$.

According to our hypothesis $k \in (1/2,1)$, we can rely on Lemma 2 1)2) for the next expansions
\begin{multline}
T^{2}\partial_{T}U_{\gamma}(T,z,\epsilon) =  T^{1-k}T^{k+1}\partial_{T}U_{\gamma}(T,z,\epsilon) \\
=
\frac{k}{(2\pi)^{1/2}}\int_{-\infty}^{+\infty} \int_{L_{\gamma}} \left( \frac{u^k}{\Gamma(\frac{1}{k}-1)}
\int_{0}^{u^k}(u^{k} - s)^{\frac{1}{k}-2} k w(s^{1/k},m,\epsilon) ds \right)
\exp( - (\frac{u}{T})^{k} ) e^{izm} \frac{du}{u} dm
\end{multline}
As a result, if one denotes $\mathcal{C}_{k}$ the operator defined as
\begin{equation}
\mathcal{C}_{k}( w(\tau,m,\epsilon) ) :=  \frac{\tau^k}{\Gamma(\frac{1}{k}-1)}
\int_{0}^{\tau^k}(\tau^{k} - s)^{\frac{1}{k}-2} k w(s^{1/k},m,\epsilon) ds \label{defin_mathcalCk}
\end{equation}
then the expression $U_{\gamma}(\frac{T}{1 + \kappa_{l}T},z,\epsilon)$ can be written as Laplace transform of order $k$ in direction
$d$ and Fourier inverse transform
\begin{equation}
U_{\gamma}(\frac{T}{1 + \kappa_{l}T},z,\epsilon) = \frac{k}{(2\pi)^{1/2}} \int_{-\infty}^{+\infty} \int_{L_{\gamma}}
(\exp( -\kappa_{l} \mathcal{C}_{k} )w)(u,m,\epsilon) \exp(- (\frac{u}{T})^{k} ) e^{izm} \frac{du}{u} dm \label{Moebius_Ugamma_int_rep}
\end{equation}
where the integrant is formally presented as a series of operators
\begin{equation}
(\exp( -\kappa_{l} \mathcal{C}_{k} )w)(\tau,m,\epsilon) := \sum_{p \geq 0} \frac{(-1)^{p} \kappa_{l}^{p}}{p!}
\mathcal{C}_{k}^{(p)}w(\tau,m,\epsilon) 
\end{equation}
and $\mathcal{C}_{k}^{(p)}$ stands for the $k-$th order iterate of the operator $\mathcal{C}_{k}$ described above.

By virtue of the identities (\ref{TkpartialTUgamma}), (\ref{TmUgamma}) and (\ref{QUgammaQUgamma}) presented in Lemma 2 and according
to the integral representation for the Moebius map acting on $U_{\gamma}$ as described above in
(\ref{Moebius_Ugamma_int_rep}), we are now in position to state
the main equation that shall fulfill the expression $w(\tau,m,\epsilon)$ provided that $U_{\gamma}(T,z,\epsilon)$ solves the
equation in prepared form (\ref{main_ivp_U_prep_form}), namely
\begin{multline}
Q(im) w(\tau,m,\epsilon) = \exp(\alpha_{D}k \tau^{k})R_{D}(im)w(\tau,z,\epsilon) \\
+ \sum_{l=1}^{D-1} \epsilon^{\Delta_{l} - d_{l} + \delta_{l}}R_{l}(im)c_{l} \sum_{n \in I_{l}} A_{l,n}(\epsilon)\\
\times
\left( \frac{\tau^{k}}{\Gamma( \frac{n + d_{l,k}}{k} )}
\int_{0}^{\tau^{k}} (\tau^{k} - s)^{\frac{n+d_{l,k}}{k}-1} k^{\delta_{l}} s^{\delta_{l}}
\left( \exp(-\kappa_{l}\mathcal{C}_{k})w \right)(s^{1/k},m,\epsilon) \frac{ds}{s} \right. \\
\left. + \sum_{1 \leq p \leq \delta_{l}-1} A_{\delta_{l},p} \frac{\tau^{k}}{\Gamma( \frac{n + d_{l,k}}{k} + \delta_{l} -p)}
\int_{0}^{\tau^{k}} (\tau^{k}-s)^{\frac{n+d_{l,k}}{k} + \delta_{l}-p-1} k^{p}s^{p}
(\exp( -\kappa_{l} \mathcal{C}_{k})w)(s^{1/k},m,\epsilon) \frac{ds}{s} \right)\\
+ 
c_{12}\frac{\tau^k}{(2\pi)^{1/2}} \int_{0}^{\tau^{k}} \int_{-\infty}^{+\infty}
Q_{1}(i(m-m_{1}))w((\tau^{k}-s)^{1/k},m-m_{1},\epsilon)\\
 \times  Q_{2}(im_{1})
w(s^{1/k},m_{1},\epsilon) \frac{1}{(\tau^{k}-s)s} dsdm_{1}
+ c_{f}\psi(\tau,m,\epsilon) \label{main_integral_eq_w}
\end{multline}

\section{Action of convolution operators on analytic and continuous function spaces}

The principal goal of this section is to present bounds for convolution maps acting on function spaces that are analytic on sectors in
$\mathbb{C}$ and continuous on $\mathbb{R}$. As in Definition 4, $S_{d}$ denotes an unbounded sector centered at 0 with bisecting direction
$d$ in $\mathbb{R}$ and $D(0,\rho) \setminus L_{-}$ stands for a cut disc centered at 0 where $L_{-} = (-\rho,0]$.

\begin{prop} Let $k \in (\frac{1}{2},1)$ be a real number. We set $\gamma_{2},\gamma_{3}$ as real numbers submitted to the next
assumption
\begin{equation}
\gamma_{2} + \gamma_{3}+2 \geq 0 \ \ , \ \ \gamma_{2}>-1 \ \ , \ \ \gamma_{3} \geq 0 \ \ , \ \ k(\gamma_{2}+\gamma_{3}+2) \in
\mathbb{N} \label{constraints_gamma_k} 
\end{equation}
Let $(\tau,m) \mapsto f(\tau,m)$ be a continuous function on $S_{d} \times \mathbb{R}$, holomorphic w.r.t $\tau$ on
$S_{d}$ for which there exists a constant $C_{1}>0$, a positive integer $N \in \mathbb{N}^{\ast}$, and real numbers $\sigma>0$, $\mu > 1$,
$\beta>0$ with
\begin{equation}
|f(\tau,m)| \leq C_{1}|\tau|^{kN} \exp( \sigma |\tau|^{k} ) (1 + |m|)^{-\mu} \exp( -\beta |m| ) \label{bds_f_taukN_exp_sigma_tauk}
\end{equation}
for all $\tau \in S_{d}$, all $m \in \mathbb{R}$. Assume moreover that for all $m \in \mathbb{R}$, the map $\tau \mapsto f(\tau,m)$
extends analytically on the cut disc $D(0,\rho)\setminus L_{-}$ and for which one can choose
a constant $C'_{1}>0$ such that
\begin{equation}
 |f(\tau,m)| \leq C_{1}'|\tau|^{k}(1 + |m|)^{-\mu}e^{-\beta |m|} \label{bds_f_tauk_near_zero}
\end{equation}
whenever $\tau \in D(0,\rho) \setminus L_{-}$ and $m \in \mathbb{R}$. 

We set
\begin{equation}
\mathcal{C}_{k,\gamma_{2},\gamma_{3}}(f)(\tau,m) = \tau^{k} \int_{0}^{\tau^{k}} (\tau^{k} -s)^{\gamma_{2}} s^{\gamma_{3}} f(s^{1/k},m) ds. 
\end{equation}
Then,\\
1) The map $(\tau,m) \mapsto \mathcal{C}_{k,\gamma_{2},\gamma_{3}}(f)(\tau,m)$ is a continuous function on
$S_{d} \times \mathbb{R}$, holomorphic w.r.t $\tau$ on $S_{d}$ for which one can sort a constant $K_{1}>0$ (depending on $\gamma_{2}$,$\sigma$)
such that
\begin{equation}
|\mathcal{C}_{k,\gamma_{2},\gamma_{3}}(f)(\tau,m)| \leq C_{1}K_{1}
|\tau|^{k(N+1)} |\tau|^{k \gamma_{3}} \exp( \sigma |\tau|^{k} ) (1 + |m|)^{-\mu} e^{-\beta |m|} \label{Ckgamma_f_bds_prop}
\end{equation}
for all $\tau \in S_{d}$, all $m \in \mathbb{R}$.\\
2) For all $m \in \mathbb{R}$, the function $\tau \mapsto \mathcal{C}_{k,\gamma_{2},\gamma_{3}}(f)(\tau,m)$ extends analytically on
$D(0,\rho) \setminus L_{-}$. Furthermore, the inequality
\begin{equation}
|\mathcal{C}_{k,\gamma_{2},\gamma_{3}}(f)(\tau,m)| \leq C_{1}'\frac{\Gamma(\gamma_{2}+1)\Gamma(\gamma_{3}+2)}{\Gamma(\gamma_{2}+\gamma_{3}+3)}
\rho^{k(\gamma_{2}+\gamma_{3}+2)}|\tau|^{k} (1 + |m|)^{-\mu}e^{-\beta |m|} \label{Ckgamma_f_bds_origin_prop}
\end{equation}
holds for all $\tau \in D(0,\rho) \setminus L_{-}$, all $m \in \mathbb{R}$.
\end{prop}
\begin{proof}
We first investigate the global behaviour of the convolution operator $\mathcal{C}_{k,\gamma_{2},\gamma_{3}}$ w.r.t $\tau$ on the unbounded
sector $S_{d}$, namely the point 1). Owing to the assumed bounds (\ref{bds_f_taukN_exp_sigma_tauk}), we get
\begin{equation}
|\mathcal{C}_{k,\gamma_{2},\gamma_{3}}(f)(\tau,m)| \leq C_{1}|\tau|^{k} \int_{0}^{|\tau|^{k}} (|\tau|^{k} -h)^{\gamma_{2}}
h^{\gamma_{3}} h^{N} \exp( \sigma h) dh (1 + |m|)^{-\mu} e^{-\beta |m|} \label{Ckgamma_f_bds1}
\end{equation}
In the next part of the proof, we need to focus on sharp upper bounds for the function
$$ G(x) = \int_{0}^{x} \exp( \sigma h) h^{\gamma_{3}+N}(x-h)^{\gamma_{2}} dh $$
We move onward as in Proposition 1 of \cite{lama2} but we need to keep track on the constants appearing in the bounds in order to provide
accurate estimates regarding the dependence with respect to the constants $\gamma_{3}$ and $N$. In accordance with the
uniform expansion $e^{\sigma h} = \sum_{n \geq 0} (\sigma h)^{n}/n!$
on every compact interval $[0,x]$, $x \geq 0$, we can write down the expansion
$$ G(x) = \sum_{n \geq 0} \frac{\sigma^{n}}{n!} \int_{0}^{x} h^{n+N+\gamma_{3}} (x-h)^{\gamma_{2}} dh $$
According to the Beta integral formula (see Appendix B from \cite{ba2}), we recall that
\begin{equation}
\int_{0}^{x} (x-h)^{\alpha - 1} h^{\beta - 1} dh = x^{\alpha + \beta - 1} \frac{\Gamma(\alpha)\Gamma(\beta)}{\Gamma( \alpha + \beta)}
\label{Beta_integral}
\end{equation}
holds for any real numbers $x \geq 0$ and $\alpha>0$, $\beta>0$. Therefore, since $N+\gamma_{3} \geq 1$ and
$\gamma_{2} > -1$, we can rewrite
$$ G(x) = \sum_{n \geq 0} \frac{\sigma^{n}}{n!} \frac{\Gamma(\gamma_{2}+1)\Gamma(n+N+1+\gamma_{3})}{\Gamma(n+N+2+\gamma_{2}+\gamma_{3})}
x^{n+1+\gamma_{2}+\gamma_{3}+N} $$
for all $x \geq 0$. On the other hand, as a consequence of the Stirling formula $\Gamma(x) \sim (2\pi)^{1/2}x^{x}e^{-x}x^{-1/2}$
as $x \rightarrow +\infty$, for any given $a>0$, there exist two constants $K_{1.1},K_{1.2}>0$ (depending on $a$) such that
\begin{equation}
\frac{K_{1.1}}{x^{a}} \leq \frac{\Gamma(x)}{\Gamma(x+a)} \leq \frac{K_{1.2}}{x^{a}} \label{quotient_Gamma}
\end{equation}
for all $x \geq 1$. As a result, we get a constants $K_{1.2}>0$ (depending on $\gamma_{2}$) for which
$$ \frac{\Gamma(n+N+1+\gamma_{3})}{\Gamma(n+N+1+\gamma_{3}+\gamma_{2}+1)} \leq
\frac{K_{1.2}}{(n+N+1+\gamma_{3})^{\gamma_{2}+1}}
\leq \frac{K_{1.2}}{(n+1)^{\gamma_{2}+1}} $$
for all $n \geq 0$. Hence, we get a constant $K_{1.3}>0$ (depending on $\gamma_{2}$)
$$ G(x) \leq K_{1.3} x^{1 + \gamma_{2}+\gamma_{3}+N} \sum_{n \geq 0} \frac{1}{(n+1)^{\gamma_{2}+1}n!} (\sigma x)^{n} $$
for all $x \geq 0$.
A second application of (\ref{quotient_Gamma}), shows the existence of a constant $K_{1.1}>0$ (depending in $\gamma_{2}$) for which
$$ \frac{1}{(n+1)^{\gamma_{2}+1}} \leq \frac{\Gamma(n+1)}{K_{1.1}\Gamma(n+\gamma_{2}+2)} $$
holds for all $n \geq 0$. Subsequently, we obtain a constant $K_{1.4}>0$ (depending on $\gamma_{2}$) such that
$$ G(x) \leq K_{1.4} x^{1+\gamma_{2}+\gamma_{3}+N} \sum_{n \geq 0} \frac{(\sigma x)^{n}}{\Gamma(n + \gamma_{2} + 2)} $$
for all $x \geq 0$.

Owing to the asymptotic property at infinity of the Wiman function
$E_{\alpha,\beta}(z) = \sum_{n \geq 0} z^{n}/\Gamma( \beta + \alpha n)$, for any $\alpha,\beta>0$ stated in \cite{erd} p. 210 we get
a constant $K_{1.5}>0$ (depending on $\gamma_{2}$,$\sigma$) with
\begin{equation}
G(x) \leq K_{1.5} x^{\gamma_{3}+N} e^{\sigma x} 
\end{equation}
for all $x \geq 0$. In accordance with this last inequality, by going back to our departing inequality
(\ref{Ckgamma_f_bds1}), we obtain the expected bounds stated in the inequality (\ref{Ckgamma_f_bds_prop}), namely
\begin{equation}
|\mathcal{C}_{k,\gamma_{2},\gamma_{3}}(f)(\tau,m)| \leq C_{1}K_{1.5}|\tau|^{k(N+1)} |\tau|^{k\gamma_{3}}
\exp( \sigma |\tau|^{k} ) (1 + |m|)^{-\mu} e^{-\beta |m|}
\end{equation}
for all $\tau \in S_{d}$, all $m \in \mathbb{R}$.\medskip

In a second part of the proof, we study local properties near the origin w.r.t $\tau$. First, we can rewrite
$\mathcal{C}_{k,\gamma_{2},\gamma_{3}}$ by using the parametrization $s=\tau^{k}u$ for $0 \leq u \leq 1$. Namely,
\begin{equation}
\mathcal{C}_{k,\gamma_{2},\gamma_{3}}(f)(\tau,m) = \tau^{k(\gamma_{2}+\gamma_{3}+2)} \int_{0}^{1} (1-u)^{\gamma_{2}} u^{\gamma_{3}}
f(\tau u^{1/k},m) du \label{C_kgamma_local_origin}
\end{equation}
holds for all $\tau \in D(0,\rho) \setminus L_{-}$ whenever $m \in \mathbb{R}$. Under the fourth assumption of (\ref{constraints_gamma_k})
and from the construction of
$f(\tau,m)$, the representation (\ref{C_kgamma_local_origin}) induces that for all $m \in \mathbb{R}$,
$\tau \mapsto \mathcal{C}_{k,\gamma_{2},\gamma_{3}}(f)(\tau,m)$ extends analytically on $D(0,\rho) \setminus L_{-}$. Furthermore,
granting to (\ref{bds_f_tauk_near_zero}), one can deduce the bounds
$$
|\mathcal{C}_{k,\gamma_{2},\gamma_{3}}(f)(\tau,m)| \leq C_{1}' |\tau|^{k} \int_{0}^{|\tau|^{k}} (|\tau|^{k} -h)^{\gamma_{2}}
h^{\gamma_{3}+1} dh (1 + |m|)^{-\mu} e^{-\beta |m|}
$$
for all $\tau \in D(0,\rho) \setminus L_{-}$, all $m \in \mathbb{R}$. With the help of (\ref{Beta_integral}), we deduce that
\begin{multline}
|\mathcal{C}_{k,\gamma_{2},\gamma_{3}}(f)(\tau,m)| \leq C_{1}'|\tau|^{k}
\frac{\Gamma(\gamma_{2}+1)\Gamma(\gamma_{3}+2)}{\Gamma( \gamma_{2}+\gamma_{3} + 3)} |\tau|^{k(\gamma_{2}+\gamma_{3}+2)}
(1+|m|)^{-\mu} e^{-\beta |m|} \\
\leq C_{1}'\frac{\Gamma(\gamma_{2}+1)\Gamma(\gamma_{3}+2)}{\Gamma(\gamma_{2}+\gamma_{3}+3)}
\rho^{k(\gamma_{2}+\gamma_{3}+2)}|\tau|^{k} (1 + |m|)^{-\mu}e^{-\beta |m|}
\end{multline}
holds when $\tau \in D(0,\rho) \setminus L_{-}$, all $m \in \mathbb{R}$ which is rephrased in
(\ref{Ckgamma_f_bds_origin_prop}).
\end{proof}

\begin{prop} Let $k \in (\frac{1}{2},1)$ be a real number. Let $(\tau,m) \mapsto f(\tau,m)$ be a continuous function on
$S_{d} \times \mathbb{R}$, holomorphic w.r.t $\tau$ on $S_{d}$ for which there exist constants $C_{2}>0$, $\nu>0$ and $\mu>1$, $\beta>0$
fulfilling
\begin{equation}
|f(\tau,m)| \leq C_{2}|\tau|^{k} \exp( \nu |\tau|^{k} ) (1 + |m|)^{-\mu} e^{-\beta |m|} 
\end{equation}
for all $\tau \in S_{d}$, all $m \in \mathbb{R}$. Take for granted that for all $m \in \mathbb{R}$, the map $\tau \mapsto f(\tau,m)$
extends analytically on the cut disc $D(0,\rho) \setminus L_{-}$ suffering the next bounds : there exists a constant $C_{2}'>0$
with
\begin{equation}
|f(\tau,m)| \leq C_{2}' |\tau|^{k} (1 + |m|)^{-\mu} e^{-\beta |m|} 
\end{equation}
for all $\tau \in D(0,\rho) \setminus L_{-}$, all $m \in \mathbb{R}$.

Let $\kappa_{l}>0$ be a real number. We consider the operator
\begin{equation}
(\exp( -\kappa_{l} \mathcal{C}_{k} )f)(\tau,m) := \sum_{p \geq 0} \frac{(-1)^{p}\kappa_{l}^{p}}{p!} \mathcal{C}_{k}^{(p)}(f)(\tau,m)
\end{equation}
where $\mathcal{C}_{k}^{(p)}$ denotes the iterate of order $p \geq 0$ of the operator $\mathcal{C}_{k}$ defined as
$$ \mathcal{C}_{k}(f)(\tau,m) = \frac{k\tau^{k}}{\Gamma(\frac{1}{k}-1)} \int_{0}^{\tau^{k}} (\tau^{k}-s)^{\frac{1}{k}-2}f(s^{1/k},m) ds
= \frac{k}{\Gamma(\frac{1}{k}-1)} \mathcal{C}_{k,\frac{1}{k}-2,0}(f)(\tau,m) $$
with the convention that $\mathcal{C}_{k}^{(0)}(f)(\tau,m) = f(\tau,m)$. Then,\\
1) The map $(\tau,m) \mapsto (\exp( -\kappa_{l} \mathcal{C}_{k} )f)(\tau,m)$ represents a continuous function on
$S_{d} \times \mathbb{R}$, holomorphic w.r.t $\tau$ on $S_{d}$ for which there exists a constant $K_{1}>0$ (depending on $k$,$\nu$) such that
\begin{equation}
|(\exp( -\kappa_{l} \mathcal{C}_{k} )f)(\tau,m)| \\
\leq C_{2}|\tau|^{k} \exp( \kappa_{l} K_{1}
\frac{k}{\Gamma(\frac{1}{k}-1)} |\tau|^{k} ) \exp( \nu |\tau|^{k} ) (1 + |m|)^{-\mu} e^{-\beta |m|} \label{bds_exp_kappaCk_Sd}
\end{equation}
for all $\tau \in S_{d}$, all $m \in \mathbb{R}$.\\
2) For all $m \in \mathbb{R}$, the function $\tau \mapsto (\exp( -\kappa_{l} \mathcal{C}_{k} )f)(\tau,m)$ extends analytically on
$D(0,\rho) \setminus L_{-}$. Furthermore,
\begin{equation}
|(\exp( -\kappa_{l} \mathcal{C}_{k} )f)(\tau,m)| \leq \exp( \frac{\kappa_{l} k \rho}{\Gamma(\frac{1}{k}+1)} )
C_{2}' |\tau|^{k} (1 + |m|)^{-\mu} e^{ -\beta |m| } \label{bds_exp_kappaCk_origin}
\end{equation}
for all $\tau \in D(0,\rho) \setminus L_{-}$, all $m \in \mathbb{R}$.
\end{prop}
\begin{proof} We need to provide estimates for each iterate $\mathcal{C}_{k}^{(N)}(f)(\tau,m)$ for $N \geq 1$. We first focus
on bounds control on the unbounded sector $S_{d}$. By induction on $N \geq 0$, with the help of the bounds
(\ref{bds_f_taukN_exp_sigma_tauk}) and (\ref{Ckgamma_f_bds_prop}) for $\gamma_{2}=\frac{1}{k}-2$ and
$\gamma_{3}=0$, we obtain a constant $K_{1}>0$ (depending on $k$,$\nu$) with
\begin{equation}
|\mathcal{C}_{k}^{(N)}(f)(\tau,m)| \leq C_{2}(\frac{k}{\Gamma(\frac{1}{k}-1)})^{N} K_{1}^{N}
|\tau|^{k(N+1)} \exp(\nu |\tau|^{k}) (1 + |m|)^{-\mu} e^{-\beta |m|} \label{bds_CkNf_onSd}
\end{equation}
for all $\tau \in S_{d}$, all $m \in \mathbb{R}$, all $N \geq 0$. Similarly, owing to (\ref{bds_f_tauk_near_zero}) and
(\ref{Ckgamma_f_bds_origin_prop})
for the same choice $\gamma_{2}=\frac{1}{k}-2$ and $\gamma_{3}=0$, we get that
\begin{equation}
|\mathcal{C}_{k}^{(N)}(f)(\tau,m)| \leq C_{2}' k^{N} (\frac{\Gamma(2)}{\Gamma(\frac{1}{k}+1)})^{N} \rho^{N}
|\tau|^{k}(1+|m|)^{-\mu} e^{-\beta |m|} \label{bds_CkNf_near_origin}
\end{equation}
for all $\tau \in D(0,\rho) \setminus L_{-}$, all $m \in \mathbb{R}$, all $N \geq 0$. 

Finally, by summing up the inequalities
(\ref{bds_CkNf_onSd}) (resp. (\ref{bds_CkNf_near_origin}) ) over $N \geq 0$ we get the forecast bounds (\ref{bds_exp_kappaCk_Sd})
(resp. (\ref{bds_exp_kappaCk_origin}) ).
\end{proof}

\begin{prop} Let $Q_{1}(X),Q_{2}(X),R(X)$ be polynomials with complex coefficients such that
\begin{equation}
\mathrm{deg}(R) \geq \mathrm{deg}(Q_{1}) \ \ , \ \ \mathrm{deg}(R) \geq \mathrm{deg}(Q_{2}) \ \ , \ \ R(im) \neq 0
\label{R>Q1_R>Q2_R_nonvanish}
\end{equation}
for all $m \in \mathbb{R}$. Take for granted that $\mu > \max( \mathrm{deg}(Q_{1})+1, \mathrm{deg}(Q_{2})+1 )$. Then, there exists
a constant $C_{3}>0$ (depending on $Q_{1},Q_{2},R,\mu,k,\nu$) for which
\begin{multline}
||\frac{1}{R(im)} \tau^{k} \int_{0}^{\tau^k} \int_{-\infty}^{+\infty} Q_{1}(i(m-m_{1}))f((\tau^{k}-s)^{1/k},m-m_{1})\\
\times 
Q_{2}(im_{1})g(s^{1/k},m_{1}) \frac{1}{(\tau^{k}-s)s} ds dm_{1} ||_{(\nu,\beta,\mu,k,\rho)} \leq 
C_{3} ||f(\tau,m)||_{(\nu,\beta,\mu,k,\rho)}||g(\tau,m)||_{(\nu,\beta,\mu,k,\rho)}
\end{multline}
holds for all $f(\tau,m),g(\tau,m) \in F_{(\nu,\beta,\mu,k,\rho)}^{d}$.
\end{prop}
\begin{proof} We proceed as in the proof of Proposition 3 of \cite{lama1}. Namely, according to the norm's definition 4, we can rewrite
\begin{multline}
B :=  ||\frac{1}{R(im)} \tau^{k} \int_{0}^{\tau^k} \int_{-\infty}^{+\infty} Q_{1}(i(m-m_{1}))f((\tau^{k}-s)^{1/k},m-m_{1})\\
\times 
Q_{2}(im_{1})g(s^{1/k},m_{1}) \frac{1}{(\tau^{k}-s)s} ds dm_{1}||_{(\nu,\beta,\mu,k,\rho)} =
\sup_{\tau \in (D(0,\rho) \setminus L_{-}) \cup S_{d}, m \in \mathbb{R}}
(1 + |m|)^{\mu} e^{\beta |m|} \frac{1 + |\tau|^{2k}}{|\tau|^{k}} \\
\times \exp(-\nu |\tau|^{k})
\times \left| \tau^{k} \int_{0}^{\tau^{k}} \int_{-\infty}^{+\infty} \{ (1 + |m-m_{1}|)^{\mu} e^{\beta |m-m_{1}|}
\frac{ 1 + |\tau^{k} - s|^{2}}{|\tau^{k} - s|} \exp( -\nu |\tau^{k} - s| ) \right. \\
\times f((\tau^{k}-s)^{1/k},m-m_{1}) \} \times \{ (1 + |m_{1}|)^{\mu} e^{\beta |m_{1}|} \frac{1 + |s|^{2}}{|s|}
\exp( -\nu |s| ) g(s^{1/k},m_{1}) \} \\
\left. \times \mathcal{B}(\tau,s,m,m_{1}) ds dm_{1} \right|
\end{multline}
for
\begin{multline*}
\mathcal{B}(\tau,s,m,m_{1}) = \frac{e^{-\beta |m-m_{1}|}e^{-\beta|m_{1}|}}{(1 + |m-m_{1}|)^{\mu}(1 + |m_{1}|)^{\mu}}
\frac{Q_{1}(i(m-m_{1})Q_{2}(im_{1})}{R(im)} \frac{|s| |\tau^{k}-s|}{(1 + |\tau^{k}-s|^{2})(1 + |s|^{2})}\\
\times
\exp(\nu|\tau^{k} -s|) \exp( \nu |s| ) \frac{1}{(\tau^{k}-s)s}.
\end{multline*}
According to the triangular inequality $|m| \leq |m-m_{1}| + |m_{1}|$ and bearing in mind the definition of the norms of $f$ and
$g$, we deduce
\begin{equation}
B \leq C_{3.1} ||f(\tau,m)||_{(\nu,\beta,\mu,k,\rho)}||g(\tau,m)||_{(\nu,\beta,\mu,k,\rho)}
\end{equation}
where
\begin{multline}
C_{3.1} = \sup_{\tau \in (D(0,\rho) \setminus L_{-}) \cup S_{d}, m \in \mathbb{R}}
(1 + |m|)^{\mu} \frac{1 + |\tau|^{2k}}{|\tau|^{k}} \exp( -\nu |\tau|^{k} )
|\tau|^{k} \\
\times \int_{0}^{|\tau|^{k}} \int_{-\infty}^{+\infty} \frac{|Q_{1}(i(m-m_{1}))||Q_{2}(im_{1})|}{R(im) (1 + |m-m_{1}|)^{\mu}
(1 + |m_{1}|)^{\mu} } \times
\frac{ \exp( \nu( |\tau|^{k} - h)) \exp( \nu h) }{(1 + (|\tau|^{k} - h)^{2})(1 + h^{2}) } dh dm_{1}.
\end{multline}
Now, we get bounds from above that can be broken up in two parts $C_{3.1} \leq C_{3.2}C_{3.3}$ where
$$
C_{3.2} = \sup_{m \in \mathbb{R}} (1 + |m|)^{\mu} \frac{1}{|R(im)|} \int_{-\infty}^{+\infty}
\frac{|Q_{1}(i(m-m_{1}))||Q_{2}(im_{1})|}{(1 + |m-m_{1}|)^{\mu} (1 + |m_{1}|)^{\mu}} dm_{1}
$$
and
$$
C_{3.3} = \sup_{\tau \in (D(0,\rho) \setminus L_{-}) \cup S_{d}} (1 + |\tau|^{2k}) \int_{0}^{|\tau|^{k}}
\frac{1}{(1 + (|\tau|^{k} - h)^{2})(1 + h^{2})} dh
$$
In the last step of the proof, we show that $C_{3.2}$ and $C_{3.3}$ have finite values. By construction, three positive constants
$\mathfrak{Q}_{1},\mathfrak{Q}_{2}$ and $\mathfrak{R}$ can be found such that
\begin{multline*}
|Q_{1}(i(m-m_{1}))| \leq \mathfrak{Q}_{1}(1 + |m-m_{1}|)^{\mathrm{deg}(Q_{1})}, \ \
|Q_{2}(im_{1})| \leq Q_{2} (1 + |m_{1}|)^{\mathrm{deg}(Q_{2})},\\ |R(im)| \geq \mathfrak{R}(1 + |m|)^{\mathrm{deg}(R)}
\end{multline*}
for all $m,m_{1} \in \mathbb{R}$. Hence,
\begin{equation}
C_{3.2} \leq \frac{ \mathfrak{Q}_{1} \mathfrak{Q}_{2} }{ \mathfrak{R} }
\sup_{m \in \mathbb{R}} (1 + |m|)^{\mu - \mathrm{deg}(R)}
\int_{-\infty}^{+\infty} \frac{1}{(1 + |m-m_{1}|)^{\mu - \mathrm{deg}(Q_{1})} (1 + |m_{1}|)^{\mu - \mathrm{deg}(Q_{2})}} dm_{1}
\end{equation}
that is finite owing to $\mu > \max( \mathrm{deg}(Q_{1})+1, \mathrm{deg}(Q_{2})+1 )$ submitted to the constraints
(\ref{R>Q1_R>Q2_R_nonvanish}) as shown in Lemma 4 from \cite{ma}. On the other hand,
\begin{multline}
C_{3.3} \leq \sup_{x \geq 0} (1 + x^{2})\int_{0}^{x} \frac{1}{(1 + (x-h)^2)(1+h^2)} dh \\
=
\sup_{x \geq 0} (1+x^{2}) 2 \frac{ \log(1 + x^{2}) + x \arctan(x) }{x (x^{2} + 4)}
\end{multline}
which is also finite.
\end{proof}

\section{Manufacturing of solutions to an auxiliary integral equation relying on a complex parameter}

The main objective of this section is the construction of a unique solution of the equation
(\ref{main_integral_eq_w}) for vanishing initial data within the Banach spaces given in Definition 4.

The first disclose further analytic assumptions on the leading polynomials $Q(X)$ and $R_{D}(X)$ in order to be able to transform our
problem (\ref{main_integral_eq_w}) into a fixed point equation as stated below, see (\ref{fixed_pt_Hepsilon}).

Namely, we take for granted that there exists a bounded sectorial annulus
$$ S_{Q,R_{D}} = \{ z \in \mathbb{C} / r_{Q,R_{D},1} \leq |z| \leq r_{Q,R_{D},2} \ \ , \ \ |\mathrm{arg}(z) - d_{Q,R_{D}}| \leq \eta_{Q,R_{D}} \} $$
with direction $d_{Q,R_{D}} \in [-\pi,\pi)$, small aperture $\eta_{Q,R_{D}}>0$ for some radii
$r_{Q,R_{D},2}> r_{Q,R_{D},1} > 1$ such that
\begin{equation}
\frac{Q(im)}{R_{D}(im)} \in S_{Q,R_{D}} \label{quotient_Q_RD_in_SQRD}
\end{equation} 
for all $m \in \mathbb{R}$. For any integer $l \in \mathbb{Z}$, we set
\begin{equation}
a_{l}(m) = \log |\frac{Q(im)}{R_{D}(im)}| + \sqrt{-1} \mathrm{arg}(\frac{Q(im)}{R_{D}(im)}) + 2l\pi\sqrt{-1}. \label{defin_alm}
\end{equation}
See Figure~\ref{fig:a} for a configuration of the points $a_{l}(m)$, $l\in\mathbb{Z}$, and the set $S_{Q,R_{D}}$ related to their definition.

\begin{figure}
	\centering
		\includegraphics[width=0.45\textwidth]{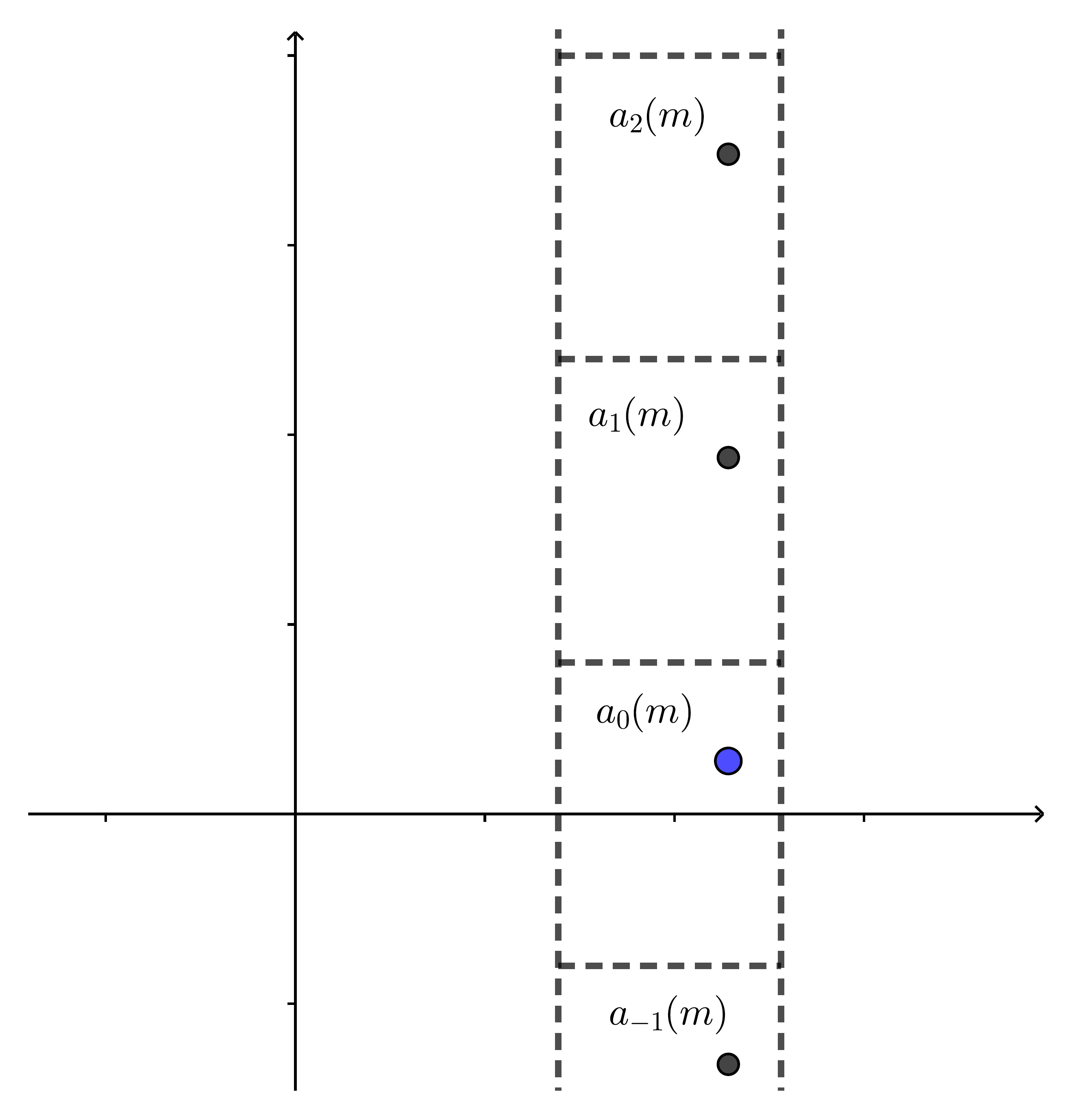}
		\includegraphics[width=0.45\textwidth]{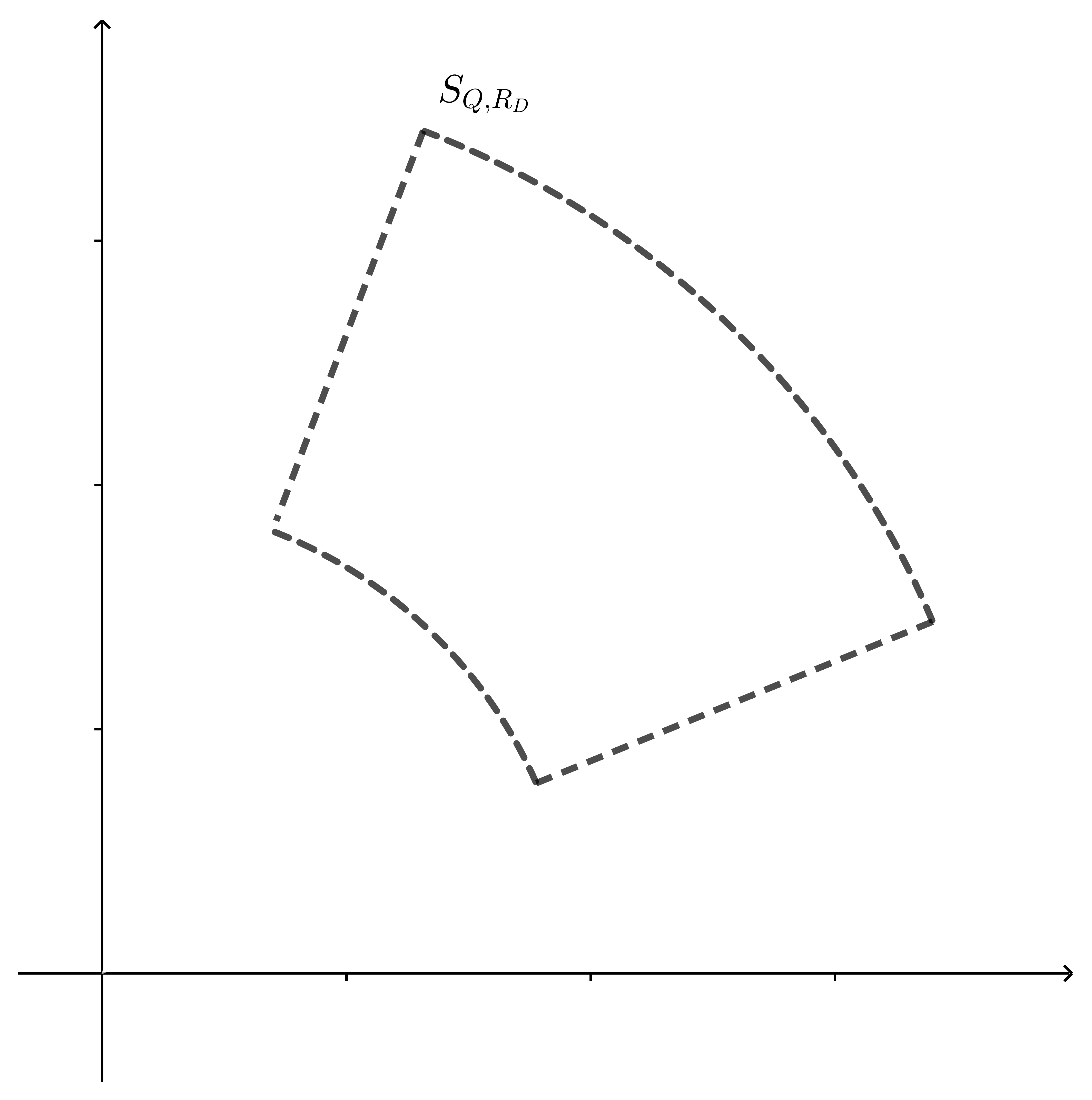}
				\caption{Configuration of $a_{l}(m)$ (left) related to $S_{Q,R_{D}}$ (right)}
				\label{fig:a}
\end{figure}

By construction, we see that
\begin{equation}
Q(im) - e^{a_{l}(m)}R_{D}(im) = 0 \label{eq_alm}
\end{equation}
for all $m \in \mathbb{R}$. Furthermore, for each $l \in \mathbb{Z}$, the equation
\begin{equation}
\alpha_{D}k\tau^{k} = a_{l}(m) \label{eq_tau_alm}
\end{equation}
possesses one solution given by
\begin{equation}
\tau_{l} = |\frac{a_{l}(m)}{\alpha_{D}k}|^{1/k} \exp( \sqrt{-1} \frac{1}{k}\mathrm{arg}(a_{l}(m))). \label{defin_taul}
\end{equation}
Indeed, by construction of $\tau^{k} = \exp(k \log(\tau))$, this equation is equivalent to write
\begin{equation}
|\tau| = |\frac{a_{l}(m)}{\alpha_{D}k}|^{1/k} \ \ , \ \ \mathrm{arg}(\tau) = \frac{\mathrm{arg}(a_{l}(m))}{k} + \frac{2h\pi}{k}
\end{equation}
for some $h \in \mathbb{Z}$. According to the hypothesis $r_{Q,R_{D},1} > 1$, we know that $|\mathrm{arg}(a_{l}(m))| < \pi/2$ and hence
\begin{equation}
|\frac{\mathrm{arg}(a_{l}(m))}{k}| < \frac{\pi}{2k} < \pi
\end{equation}
since we assume that $\frac{1}{2} < k < 1$. Owing to the fact that $\mathrm{arg}(\tau)$ belongs to $(-\pi,\pi)$, it forces $h=0$ and hence
$\mathrm{arg}(\tau) = \mathrm{arg}(a_{l}(m))/k$.

We consider the set
$$\Theta_{Q,R_{D}} = \{ \frac{\mathrm{arg}(a_{l}(m))}{k} / m \in \mathbb{R}, l \in \mathbb{Z} \}$$
of so-called forbidden directions. We choose the aperture $\eta_{Q,R_{D}}>0$ small enough in a way that for all directions
$d \in (-\pi/2,\pi/2) \setminus \Theta_{Q,R_{D}}$, we can find some unbounded sector $S_{d}$ centered at 0 with small aperture
$\delta_{S_{d}}>0$ and
bisecting direction $d$ such that $\tau_{l} \notin S_{d} \cup D(0,\rho)$ for some fixed $\rho>0$ small enough and for all $l \in \mathbb{Z}$.

For all $\tau \in \mathbb{C} \setminus \mathbb{R}_{-}$, all $m \in \mathbb{R}$, we consider the function
\begin{equation}
H(\tau,m) = Q(im) - \exp( \alpha_{D}k\tau^{k} )R_{D}(im) \label{defin_Htaum} 
\end{equation}
Let $d \in (-\pi/2,\pi/2) \setminus \Theta_{Q,R_{D}}$ and take a sector $S_{d}$ and a disc $D(0,\rho)$ as above.

{\bf 1)} Our first goal is to provide lower bounds for the function $|H(\tau,m)|$ when $\tau \in S_{d}$ and $m \in \mathbb{R}$.
Let $\tau \in S_{d}$. Then, we can write
\begin{equation}
\tau = \tau_{l}re^{\sqrt{-1}\theta} \label{factor_tau}
\end{equation}
for some well chosen $l \in \mathbb{Z}$, where $r \geq 0$ and where $\theta$ belong to some small interval $I_{S_{d}}$ which is close to
0 but such that $0 \notin I_{S_d}$. In particular, we choose $I_{S_{d}}$ in a way that $\mathrm{arg}(\tau_{l}) + \theta$ belongs to
$(-\pi,\pi)$ for all $\theta \in I_{S_{d}}$.

Hence, owing to the fact that $\tau_{l}$ solves (\ref{eq_tau_alm}), we can rewrite
$$ \alpha_{D}k \tau^{k} - a_{l}(m) = \alpha_{D}k \tau_{l}^{k} r^{k} e^{\sqrt{-1}k\theta} - a_{l}(m) = a_{l}(m)(r^{k}e^{\sqrt{-1}k\theta} - 1) $$
In particular, if the radius $r_{Q,R_{D},2}>r_{Q,R_{D},1}$ is chosen close enough to $r_{Q,R_{D},1}$, we get a constant $\eta_{1,l}>0$
(depending on $l$) for which
\begin{equation}
|\alpha_{D}k \tau^{k} - a_{l}(m) - \sqrt{-1}h2\pi| \geq \eta_{1,l} \label{lowbds_eq_tauk_alm}
\end{equation}
for all $h \in \mathbb{Z}$, all $\tau \in S_{d}$, all $m \in \mathbb{R}$. Namely, for each $m \in \mathbb{R}$, the set
$$ \mathcal{L}_{l,m} = \{ a_{l}(m)(x e^{\sqrt{-1}k\theta} - 1) / x \geq 0 \}$$
represents a halfline passing through the point $-a_{l}(m)$ and
close to the origin in $\mathbb{C}$ and consequently it avoids the set of points $\{ \sqrt{-1}h 2\pi / h \in \mathbb{Z} \}$.

Figure~\ref{fig:b} illustrates a configuration of some of the halflines described in the construction.

\begin{figure}
	\centering
		\includegraphics[width=0.3\textwidth]{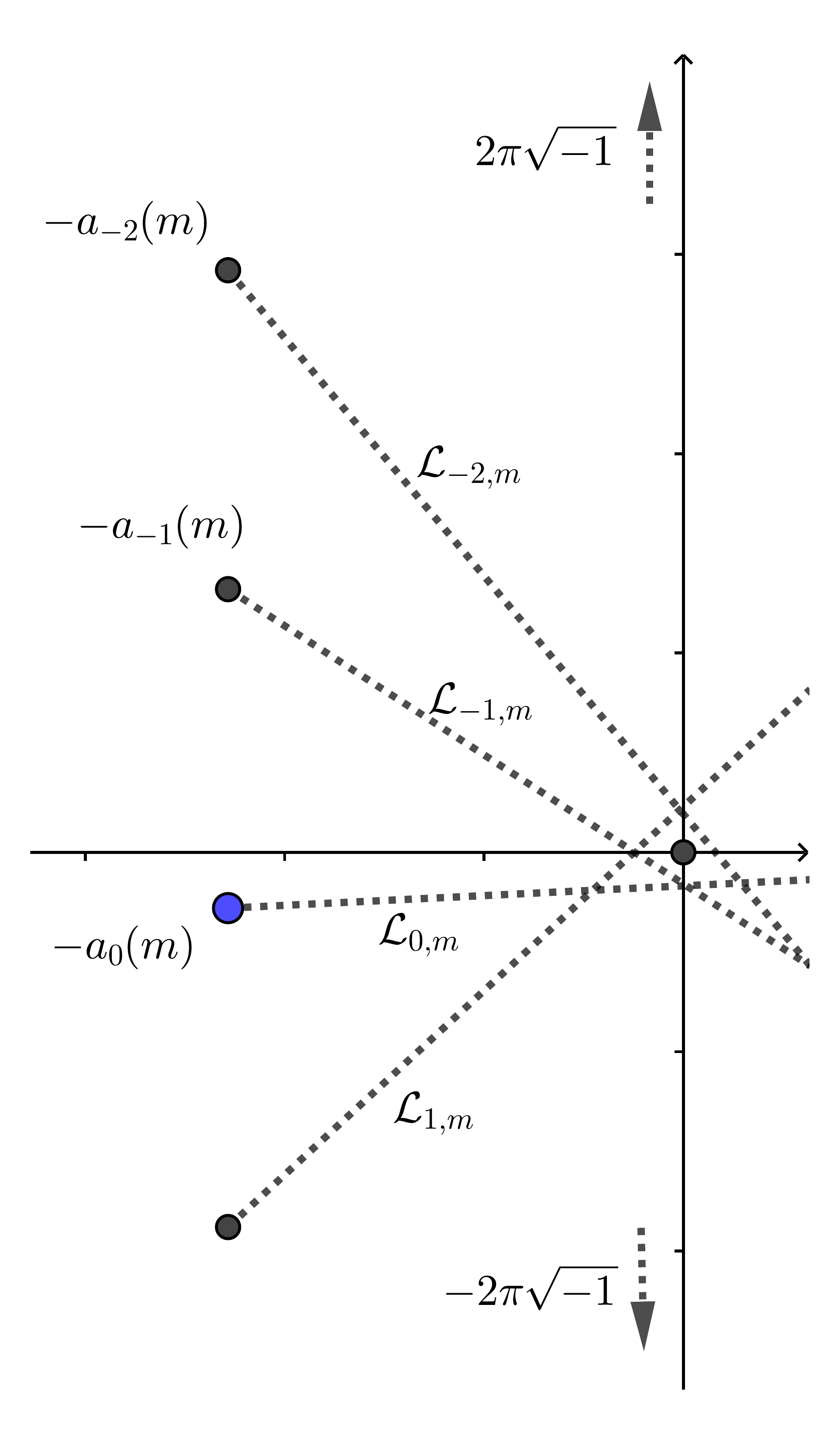}
				\caption{Some of the halflines $\mathcal{L}_{l,m}$}
				\label{fig:b}
\end{figure}

Now, owing to the equality (\ref{eq_alm}), we can rewrite
\begin{multline}
H(\tau,m) = Q(im) - \exp( \alpha_{D}k \tau^{k} - a_{l}(m) ) \exp( a_{l}(m) ) R_{D}(im) \\
= Q(im)
(1 - \exp( \alpha_{D}k \tau^{k} - a_{l}(m) )) \label{factor_H_Q_exp} 
\end{multline}
According to (\ref{lowbds_eq_tauk_alm}), we obtain a constant $\eta_{2,l}>0$ (depending on $l$) for which
\begin{equation}
|H(\tau,m)| \geq |Q(im)|\eta_{2,l} \label{low_bds_Htaum_first} 
\end{equation}
for all $\tau \in S_{d}$, all $m \in \mathbb{R}$.

In a second step, we aim attention at lower bounds for large values of $|\tau|$ on $S_{d}$. We first carry out some preliminary computations,
namely we need to expand
\begin{multline}
\mathrm{Re}( a_{l}(m) (r^{k}e^{\sqrt{-1}k \theta} - 1) ) \\
= r^{k} \left( \log |\frac{Q(im)}{R_{D}(im)}| \cos(k \theta)
- (\mathrm{arg}( \frac{Q(im)}{R_{D}(im)} ) + 2l\pi) \sin(k \theta) \right) - \log |\frac{Q(im)}{R_{D}(im)}| \label{expand_Re_alm_exp}
\end{multline}
We assume that the segment $I_{S_{d}}$ is close enough to 0 in a way that we can find a constant $\Delta_{1}>0$ submitted to the next
inequality
\begin{equation}
\log |\frac{Q(im)}{R_{D}(im)}| \cos(k \theta)
- (\mathrm{arg}( \frac{Q(im)}{R_{D}(im)} ) + 2l\pi) \sin(k \theta) \geq \Delta_{1} \label{low_bds_QRD_cos_sin}
\end{equation}
for all $m \in \mathbb{R}$, all $\theta \in I_{S_{d}}$. Besides, according to the inclusion (\ref{quotient_Q_RD_in_SQRD}), we notice that
\begin{equation}
0< \log(r_{Q,R_{D},1}) \leq \log |\frac{Q(im)}{R_{D}(im)}| \leq \log(r_{Q,R_{D},2}) \label{bds_log_QRD}
\end{equation}
holds for all $m \in \mathbb{R}$. As a result, collecting (\ref{expand_Re_alm_exp}), (\ref{low_bds_QRD_cos_sin}) and
(\ref{bds_log_QRD}) yields the lower bounds
\begin{equation}
\mathrm{Re}( a_{l}(m) (r^{k}e^{\sqrt{-1}k \theta} - 1) ) \geq \Delta_{1}r^{k} - \log( r_{Q,R_{D},2} ) \label{low_bds_Re_alm_exp}
\end{equation}
for all $r \geq 0$, all $\theta \in I_{S_{d}}$, all $m \in \mathbb{R}$.

Departing from the factorization (\ref{factor_H_Q_exp}) we get the next estimates from below
\begin{multline}
|H(\tau,m)| \geq |Q(im)| \left| 1 - |\exp \left( a_{l}(m) (r^{k}e^{\sqrt{-1} k \theta} - 1) \right)| \right| \\
=
|Q(im)| \left| 1 - \exp \left( \mathrm{Re}( a_{l}(m)(r^{k} e^{\sqrt{-1} k \theta} - 1) ) \right) \right|\\
= |Q(im)| \exp \left( \mathrm{Re}( a_{l}(m)(r^{k} e^{\sqrt{-1} k \theta} - 1) ) \right) \left| 1 -
\exp \left( -\mathrm{Re}( a_{l}(m)(r^{k} e^{\sqrt{-1} k \theta} - 1) ) \right) \right| \label{low_bds_Htaum_factor}
\end{multline}
for all $r \geq 0$, all $\theta \in I_{S_{d}}$, all $m \in \mathbb{R}$. We select a real number $r_{1}>0$ large enough such that
\begin{equation}
\exp( -(\Delta_{1}r^{k} - \log( r_{Q,R_{D},2} ) ) ) \leq 1/2 \label{exp_Delta1_QRD_small} 
\end{equation}
for all $r \geq r_{1}$. Under this last constraint (\ref{exp_Delta1_QRD_small}), we deduce from (\ref{low_bds_Re_alm_exp}) and
(\ref{low_bds_Htaum_factor}) that
$$
|H(\tau,m)| \geq \frac{1}{2}|Q(im)| \exp \left( \Delta_{1}r^{k} - \log(r_{Q,R_{D},2}) \right)
$$
for all $r \geq r_{1}$, all $\theta \in I_{S_{d}}$, all $m \in \mathbb{R}$. Now, in view of the decomposition (\ref{factor_tau}), we get in
particular that $|\tau|=r |\tau_{l}|$. Consequently, we see that
\begin{equation}
|H(\tau,m)| \geq \frac{1}{2}|Q(im)| \exp \left( \frac{\Delta_{1}}{|\tau_{l}|^{k}} |\tau|^{k} - \log(r_{Q,R_{D},2}) \right)
\label{low_bds_Htaum_large_tau}
\end{equation}
for all $\tau \in S_{d}$ with $|\tau| \geq r_{1}|\tau_{l}|$.

As a result, gathering (\ref{low_bds_Htaum_first}) and (\ref{low_bds_Htaum_large_tau}), together with the shape of $a_{l}(m)$ and
$\tau_{l}$ given in (\ref{defin_alm}), (\ref{defin_taul}), we obtain two constants $A_{H,d},B_{H,d}>0$ depending on
$k,S_{Q,R_{D}},S_{d}$ for which
\begin{equation}
|H(\tau,m)| \geq A_{H,d}|Q(im)|\exp(B_{H,d} \alpha_{D}|\tau|^{k}) \label{low_bds_Htaum_on_sector}
\end{equation}
for all $\tau \in S_{d}$, all $m \in \mathbb{R}$.\medskip

{\bf 2)} In a second step, we display lower bounds when $\tau$ belongs to the cut disc $D(0,\rho) \setminus L_{-}$ where
$L_{-} = (-\rho,0]$. Let $\tau = r e^{\sqrt{-1} \theta}$ for some $\theta \in (-\pi,\pi)$ and $0 < r < \rho$.
Let $l \in \mathbb{Z}$. We first compute the
real part
$$
\mathrm{Re}( \alpha_{D}kr^{k} e^{\sqrt{-1} k \theta} - a_{l}(m) ) = \alpha_{D}kr^{k}\cos(k \theta) - \log |\frac{Q(im)}{R_{D}(im)}| 
$$
for all $m \in \mathbb{R}$. Therefore, owing to (\ref{bds_log_QRD}), we may select $r_{Q,R_{D},1}>0$ large enough such that
\begin{equation}
\exp \left( \mathrm{Re}( \alpha_{D}kr^{k} e^{\sqrt{-1} k \theta} - a_{l}(m) ) \right) \leq \exp \left( \alpha_{D}k \rho^{k}
- \log( r_{Q,R_{D},1} ) \right) \leq \frac{1}{2}   
\end{equation}
for all $\theta \in (-\pi,\pi)$, $0 < r < \rho$ and $m \in \mathbb{Z}$. Hence, thanks to the factorization (\ref{factor_H_Q_exp})
it follows that
\begin{multline}
|H(\tau,m)| \geq |Q(im)| \left| 1 - | \exp \left( \alpha_{D}k \tau^{k} - a_{l}(m) \right)| \right|
\\
= |Q(im)| \left| 1 - \exp \left( \mathrm{Re}( \alpha_{D}kr^{k} e^{\sqrt{-1} k \theta} - a_{l}(m) ) \right) \right|
\geq \frac{1}{2} |Q(im)| \label{low_bds_Htaum_on_cdisc}
\end{multline}
for all $\tau \in D(0,\rho) \setminus L_{-}$, all $m \in \mathbb{R}$.\medskip

In the next proposition we provide sufficient conditions for which the equation (\ref{main_integral_eq_w}) possesses a solution
$w^{d}(\tau,m,\epsilon)$ within the Banach space $F^{d}_{(\nu,\beta,\mu,k,\rho)}$.

\begin{prop}
We make the next additional assumptions
\begin{equation}
B_{H,d} \alpha_{D} > \kappa_{l}K_{1} \frac{k}{\Gamma( \frac{1}{k} - 1)}  \label{cond_Bhd_alphaDK1}
\end{equation}
for all $1 \leq l \leq D-1$, where $K_{1}$ is a constant depending on $k,\nu$ defined in Proposition 2 1) and $B_{H,d}$ is selected in
(\ref{low_bds_Htaum_on_sector}). Under the condition that the moduli $|c_{12}|$,$|c_{f}|$ and $|c_{l}|$ for $1 \leq l \leq D-1$ are chosen
small enough, we can find a constant $\varpi>0$ for which the equation
(\ref{main_integral_eq_w}) has a unique solution $w^{d}(\tau,m,\epsilon)$ in the space
$F^{d}_{(\nu,\beta,\mu,k,\rho)}$ controlled in norm in a way that
$||w^{d}(\tau,m,\epsilon)||_{(\nu,\beta,\mu,k,\rho)} \leq \varpi$ for all $\epsilon \in D(0,\epsilon_{0})$, where $\beta,\mu>0$ are chosen as
in (\ref{norm_beta_mu_F_n_a_ln}), $\nu > 1$ is taken as in Lemma 1, the sector $S_{d}$ and the disc $D(0,\rho)$ are suitably selected in a way
that
$\tau_{l} \notin S_{d} \cup D(0,\rho)$ for all $l \in \mathbb{Z}$ where $\tau_{l}$ is displayed by (\ref{defin_taul}) as described above.
\end{prop}

\begin{proof}
We initiate the proof with a lemma that introduces a map related to (\ref{main_integral_eq_w}) and describes some of its properties that will
allow us to apply a fixed point theorem for it.

\begin{lemma} One can sort the moduli $|c_{12}|$,$|c_{f}|$ and $|c_{l}|$ for $1 \leq l \leq D-1$ tiny in size for which a constant
$\varpi>0$ can be picked up in a way
that the map $\mathcal{H}_{\epsilon}$ defined as
\begin{multline}
\mathcal{H}_{\epsilon}(w(\tau,m)) := 
\sum_{l=1}^{D-1} \epsilon^{\Delta_{l} - d_{l} + \delta_{l}}R_{l}(im) c_{l}\sum_{n \in I_{l}} A_{l,n}(\epsilon)\\
\times
\left( \frac{\tau^{k}}{H(\tau,m)\Gamma( \frac{n + d_{l,k}}{k} )}
\int_{0}^{\tau^{k}} (\tau^{k} - s)^{\frac{n+d_{l,k}}{k}-1} k^{\delta_{l}} s^{\delta_{l}}
\left( \exp(-\kappa_{l}\mathcal{C}_{k})w \right)(s^{1/k},m) \frac{ds}{s} \right. \\
+ \sum_{1 \leq p \leq \delta_{l}-1} A_{\delta_{l},p} \frac{\tau^{k}}{H(\tau,m)\Gamma( \frac{n + d_{l,k}}{k} + \delta_{l} -p)}
\int_{0}^{\tau^{k}} (\tau^{k}-s)^{\frac{n+d_{l,k}}{k} + \delta_{l}-p-1}\\
\left. \times k^{p}s^{p}
(\exp( -\kappa_{l} \mathcal{C}_{k})w)(s^{1/k},m) \frac{ds}{s} \right)\\
+ 
c_{12}\frac{\tau^k}{(2\pi)^{1/2}H(\tau,m)} \int_{0}^{\tau^{k}} \int_{-\infty}^{+\infty}
Q_{1}(i(m-m_{1}))w((\tau^{k}-s)^{1/k},m-m_{1})\\
 \times  Q_{2}(im_{1})
w(s^{1/k},m_{1}) \frac{1}{(\tau^{k}-s)s} dsdm_{1}
+ c_{f}\frac{\psi(\tau,m,\epsilon)}{H(\tau,m)}
\end{multline}
fulfills the next features:\medskip

\noindent 1) The next inclusion
\begin{equation}
\mathcal{H}_{\epsilon}( \bar{B}(0,\varpi) ) \subset \bar{B}(0,\varpi) \label{H_epsilon_inclusion}
\end{equation}
holds where $\bar{B}(0,\varpi)$ represents the closed ball of radius $\varpi>0$ centered at 0 in
$F^{d}_{(\nu,\beta,\mu,k,\rho)}$ for all $\epsilon \in D(0,\epsilon_{0})$.\medskip

\noindent 2) The shrinking condition
\begin{equation}
|| \mathcal{H}_{\epsilon}(w_{2}) - \mathcal{H}_{\epsilon}(w_{1}) ||_{(\nu,\beta,\mu,k,\rho)} \leq \frac{1}{2}
||w_{2} - w_{1}||_{(\nu,\beta,\mu,k,\rho)} \label{H_epsilon_shrink}
\end{equation}
occurs whenever $w_{1},w_{2} \in \bar{B}(0,\varpi)$, for all $\epsilon \in D(0,\epsilon_{0})$.
\end{lemma}
\begin{proof} Foremost, we focus on the first property (\ref{H_epsilon_inclusion}). Let $w(\tau,m)$ belonging to
$F^{d}_{(\nu,\beta,\mu,k,\rho)}$. We take $\epsilon \in D(0,\epsilon_{0})$ and set $\varpi>0$ such that
$||w(\tau,m)||_{(\nu,\beta,\mu,k,\rho)} \leq \varpi$. In particular, we notice that the next estimates
\begin{equation}
|w(\tau,m)| \leq \varpi |\tau|^{k} e^{\nu |\tau|^{k}} e^{-\beta |m|} (1 + |m|)^{-\mu} \label{bds_w}
\end{equation}
hold for all $\tau \in S_{d} \cup (D(0,\rho) \setminus L_{-})$. As a consequence of Proposition 2, we get that
$(\tau,m) \mapsto \exp( -\kappa_{l} \mathcal{C}_{k} )(w)(\tau,m)$ defines a continuous function on $S_{d} \times \mathbb{R}$,
holomorphic w.r.t $\tau$ on $S_{d}$ and a constant $K_{1}>0$ (depending on $k,\nu$) can be found such that
\begin{equation}
|(\exp( -\kappa_{l} \mathcal{C}_{k} )w)(\tau,m)| \leq \varpi |\tau|^{k}
\exp( \kappa_{l}K_{1} \frac{k}{\Gamma( \frac{1}{k} - 1)} |\tau|^{k} ) \exp( \nu |\tau|^{k}) (1 + |m|)^{-\mu}
e^{-\beta |m|} \label{bds_exp_kappaCk_Sd_w}
\end{equation}
for all $\tau \in S_{d}$, all $m \in \mathbb{R}$. Furthermore, the application of Proposition 1 for
$\gamma_{2} = \frac{n + d_{l,k}}{k} - 1$, $\gamma_{3} = \delta_{l}-1$ with $n \in I_{l}$ grants a constant
$C_{4}>0$ (depending on $I_{l}$,$k$,$\kappa_{l}$,$d_{l}$,$\delta_{l}$,$\nu$) with
\begin{multline*}
|\tau^{k}\int_{0}^{\tau^k} (\tau^{k}-s)^{\frac{n + d_{l,k}}{k}-1} s^{\delta_l} (\exp( -\kappa_{l} \mathcal{C}_{k} )w)(s^{1/k},m) \frac{ds}{s}|
\\
\leq \varpi C_{4} |\tau|^{2k + k(\delta_{l}-1)}
\exp( \kappa_{l}K_{1} \frac{k}{\Gamma( \frac{1}{k} - 1)} |\tau|^{k} ) \exp( \nu |\tau|^{k}) (1 + |m|)^{-\mu} e^{-\beta |m|}
\end{multline*}
provided that $\tau \in S_{d}$, $m \in \mathbb{R}$. Looking back to the lower bounds
(\ref{low_bds_Htaum_on_sector}) and having a glance at the constraints (\ref{cond_Bhd_alphaDK1})
allows us to reach the estimates
\begin{multline}
\left| \frac{R_{l}(im) \tau^{k}}{H(\tau,m)} \int_{0}^{\tau^{k}} (\tau^{k} - s)^{\frac{n + d_{l,k}}{k} - 1} s^{\delta_l}
(\exp( - \kappa_{l} \mathcal{C}_{k} )w)(s^{1/k},m) \frac{ds}{s} \right| \\
\leq \varpi \frac{C_{4}}{A_{H,d}}
\left|\frac{R_{l}(im)}{Q(im)} \right| |\tau|^{k(\delta_{l}+1)}
\exp \left( (\kappa_{l}K_{1}\frac{k}{\Gamma(\frac{1}{k}-1)} - B_{H,d} \alpha_{D})|\tau|^{k} \right)\\
\times
\exp( \nu |\tau|^{k} ) (1 + |m|)^{-\mu} e^{-\beta |m|} \leq \varpi \frac{C_{4}}{A_{H,d}}
\sup_{m \in \mathbb{R}} \left| \frac{R_{l}(im)}{Q(im)} \right|
\left( \sup_{|\tau| \geq 0} |\tau|^{k \delta_{l}} (1 + |\tau|^{2k}) \right. \\
\left. \times 
\exp ( (\kappa_{l}K_{1}\frac{k}{\Gamma(\frac{1}{k}-1)} - B_{H,d} \alpha_{D})|\tau|^{k}) \right) \times
\frac{|\tau|^{k}}{1 + |\tau|^{2k}} \exp( \nu |\tau|^{k} ) (1 + |m|)^{-\mu} \exp( -\beta |m| ) \label{bds_op1_sector_inclusion_Hepsilon}
\end{multline}
whenever $\tau \in S_{d}$, $m \in \mathbb{R}$.

On the other hand, Proposition 2 guarantees that for all $m \in \mathbb{R}$, the function \\
$\tau \mapsto (\exp( -\kappa_{l} \mathcal{C}_{k} )w)(\tau,m)$ extends analytically on $D(0,\rho) \setminus L_{-}$ with the bounds
\begin{equation}
|(\exp( -\kappa_{l} \mathcal{C}_{k} )w)(\tau,m)| \leq \exp( \frac{\kappa_{l} k \rho}{\Gamma( \frac{1}{k} + 1)} )
\varpi \exp( \nu \rho^{k} ) |\tau|^{k} (1 + |m|)^{-\mu} e^{-\beta |m|} \label{bds_exp_kappaCk_origin_w}
\end{equation}
for all $\tau \in D(0,\rho) \setminus L_{-}$, all $m \in \mathbb{R}$. As a consequence, Proposition 1 specialized for
$\gamma_{2} = \frac{n + d_{l,k}}{k} - 1$, $\gamma_{3} = \delta_{l}-1$ with $n \in I_{l}$ gives raise to a constant
$C_{4}'>0$ (depending on $I_{l}$,$k$,$\kappa_{l}$,$d_{l}$,$\delta_{l}$,$\nu$,$\rho$) for which
$$
|\tau^{k}\int_{0}^{\tau^k} (\tau^{k}-s)^{\frac{n + d_{l,k}}{k}-1} s^{\delta_l} (\exp( -\kappa_{l} \mathcal{C}_{k} )w)(s^{1/k},m) \frac{ds}{s}|
\leq \varpi C_{4}' |\tau|^{k} (1 + |m|)^{-\mu} e^{-\beta |m|}
$$
provided that $\tau \in D(0,\rho) \setminus L_{-}$, $m \in \mathbb{R}$. Keeping in mind the lower bounds
(\ref{low_bds_Htaum_on_cdisc}) we notice that
\begin{multline}
\left| \frac{R_{l}(im) \tau^{k}}{H(\tau,m)} \int_{0}^{\tau^{k}} (\tau^{k} - s)^{\frac{n + d_{l,k}}{k} - 1} s^{\delta_l}
(\exp( - \kappa_{l} \mathcal{C}_{k} )w)(s^{1/k},m) \frac{ds}{s} \right| \\
\leq 2 \varpi C_{4}' \left| \frac{R_{l}(im)}{Q(im)} \right| |\tau|^{k} (1 + |m|)^{-\mu} e^{-\beta |m|}
\leq 2 \varpi C_{4}' \sup_{m \in \mathbb{R}} \left| \frac{R_{l}(im)}{Q(im)} \right| (1 + \rho^{2k})
\frac{|\tau|^{k}}{1 + |\tau|^{2k}} \exp( \nu |\tau|^{k} )\\
\times (1 + |m|)^{-\mu} \exp( -\beta |m| ) \label{bds_op1_cdisc_inclusion_Hepsilon}
\end{multline}
for all $\tau \in D(0,\rho) \setminus L_{-}$, all $m \in \mathbb{R}$.

By clustering (\ref{bds_op1_sector_inclusion_Hepsilon}) and (\ref{bds_op1_cdisc_inclusion_Hepsilon}), we conclude that there exists a
constant $C_{5}>0$ (depending on $I_{l},k,\kappa_{l},d_{l},\delta_{l},\nu,\rho,S_{Q,R_{D}},S_{d},R_{l},Q$) with
\begin{equation}
|| \frac{R_{l}(im) \tau^{k}}{H(\tau,m)} \int_{0}^{\tau^{k}} (\tau^{k} - s)^{\frac{n + d_{l,k}}{k} - 1} s^{\delta_l}
(\exp( - \kappa_{l} \mathcal{C}_{k} )w)(s^{1/k},m) \frac{ds}{s} ||_{(\nu,\beta,\mu,k,\rho)}
\leq C_{5} \varpi. \label{norm_bds_conv_exp_op_1}
\end{equation}
Keeping in view the bounds (\ref{bds_exp_kappaCk_Sd_w}), an application of Proposition 1 for
$$ \gamma_{2} = \frac{n + d_{l,k}}{k} + \delta_{l} -p - 1 \ \ , \ \ \gamma_{3} = p-1 $$
where $n \in I_{l}$, with $1 \leq p \leq \delta_{l}-1$ yields a constant $C_{6}>0$
(depending on $I_{l}$,$k$,$\kappa_{l}$,$d_{l}$,$\delta_{l}$,$\nu$) with
\begin{multline*}
|\tau^{k}\int_{0}^{\tau^k} (\tau^{k}-s)^{\frac{n + d_{l,k}}{k} + \delta_{l} - p - 1} s^{p}
(\exp( -\kappa_{l} \mathcal{C}_{k} )w)(s^{1/k},m) \frac{ds}{s}|
\\
\leq \varpi C_{6} |\tau|^{2k + k(p-1)}
\exp( \kappa_{l}K_{1} \frac{k}{\Gamma( \frac{1}{k} - 1)} |\tau|^{k} ) \exp( \nu |\tau|^{k}) (1 + |m|)^{-\mu} e^{-\beta |m|}
\end{multline*}
for all $\tau \in S_{d}$, $m \in \mathbb{R}$ and $1 \leq p \leq \delta_{l}-1$.

Owing to the lower bounds (\ref{low_bds_Htaum_on_sector}) under the restriction (\ref{cond_Bhd_alphaDK1}), we deduce that
\begin{multline}
\left| \frac{R_{l}(im) \tau^{k}}{H(\tau,m)} \int_{0}^{\tau^{k}} (\tau^{k} - s)^{\frac{n + d_{l,k}}{k} + \delta_{l} -p-1} s^{p}
(\exp( - \kappa_{l} \mathcal{C}_{k} )w)(s^{1/k},m) \frac{ds}{s} \right| \\
\leq \varpi \frac{C_{6}}{A_{H,d}}
\left|\frac{R_{l}(im)}{Q(im)} \right| |\tau|^{k(p+1)}
\exp \left( (\kappa_{l}K_{1}\frac{k}{\Gamma(\frac{1}{k}-1)} - B_{H,d} \alpha_{D})|\tau|^{k} \right)\\
\times
\exp( \nu |\tau|^{k} ) (1 + |m|)^{-\mu} e^{-\beta |m|} \leq \varpi \frac{C_{6}}{A_{H,d}}
\sup_{m \in \mathbb{R}} \left| \frac{R_{l}(im)}{Q(im)} \right|
\left( \sup_{|\tau| \geq 0} |\tau|^{k p} (1 + |\tau|^{2k}) \right. \\
\left. \times 
\exp ( (\kappa_{l}K_{1}\frac{k}{\Gamma(\frac{1}{k}-1)} - B_{H,d} \alpha_{D})|\tau|^{k}) \right) \times
\frac{|\tau|^{k}}{1 + |\tau|^{2k}} \exp( \nu |\tau|^{k} ) (1 + |m|)^{-\mu} \exp( -\beta |m| ) \label{bds_op2_sector_inclusion_Hepsilon}
\end{multline}
whenever $\tau \in S_{d}$, $m \in \mathbb{R}$ with $1 \leq p \leq \delta_{l}-1$.

Using the bounds (\ref{bds_exp_kappaCk_origin_w}) and calling up Proposition 1 for the values
$$ \gamma_{2} = \frac{n + d_{l,k}}{k} + \delta_{l} -p - 1 \ \ , \ \ \gamma_{3} = p-1 $$
where $n \in I_{l}$, with $1 \leq p \leq \delta_{l}-1$, we obtain a constant $C_{6}'>0$ 
(depending on $I_{l}$,$k$,$\kappa_{l}$,$d_{l}$,$\delta_{l}$,$\nu$,$\rho$) for which
$$
|\tau^{k}\int_{0}^{\tau^k} (\tau^{k}-s)^{\frac{n + d_{l,k}}{k} + \delta_{l} - p -1}
s^{p} (\exp( -\kappa_{l} \mathcal{C}_{k} )w)(s^{1/k},m) \frac{ds}{s}|
\leq \varpi C_{6}' |\tau|^{k} (1 + |m|)^{-\mu} e^{-\beta |m|}
$$
for $\tau \in D(0,\rho) \setminus L_{-}$, $m \in \mathbb{R}$. With the help of the lower bounds
(\ref{low_bds_Htaum_on_cdisc}), we deduce
\begin{multline}
\left| \frac{R_{l}(im) \tau^{k}}{H(\tau,m)} \int_{0}^{\tau^{k}} (\tau^{k} - s)^{\frac{n + d_{l,k}}{k} + \delta_{l} -p-1} s^{p}
(\exp( - \kappa_{l} \mathcal{C}_{k} )w)(s^{1/k},m) \frac{ds}{s} \right| \\
\leq 2 \varpi C_{6}' \left| \frac{R_{l}(im)}{Q(im)} \right| |\tau|^{k} (1 + |m|)^{-\mu} e^{-\beta |m|}
\leq 2 \varpi C_{6}' \sup_{m \in \mathbb{R}} \left| \frac{R_{l}(im)}{Q(im)} \right| (1 + \rho^{2k})
\frac{|\tau|^{k}}{1 + |\tau|^{2k}} \exp( \nu |\tau|^{k} )\\
\times (1 + |m|)^{-\mu} \exp( -\beta |m| ) \label{bds_op2_cdisc_inclusion_Hepsilon}
\end{multline}
provided that $\tau \in D(0,\rho) \setminus L_{-}$ and $m \in \mathbb{R}$.

By grouping (\ref{bds_op2_sector_inclusion_Hepsilon}) and (\ref{bds_op2_cdisc_inclusion_Hepsilon}), we deduce the existence of a
constant $C_{7}>0$ (depending on $I_{l},k,\kappa_{l},d_{l},\delta_{l},\nu,\rho,S_{Q,R_{D}},S_{d},R_{l},Q$) with
\begin{equation}
|| \frac{R_{l}(im) \tau^{k}}{H(\tau,m)} \int_{0}^{\tau^{k}} (\tau^{k} - s)^{\frac{n + d_{l,k}}{k} + \delta_{l} -p -1 } s^{p}
(\exp( - \kappa_{l} \mathcal{C}_{k} )w)(s^{1/k},m) \frac{ds}{s} ||_{(\nu,\beta,\mu,k,\rho)}
\leq C_{7} \varpi. \label{norm_bds_conv_exp_op_2}
\end{equation}

On the other hand, taking into account the assumption (\ref{constraints_degree_coeff}) and the lower bounds
(\ref{low_bds_Htaum_first}) together with (\ref{low_bds_Htaum_on_cdisc}), the application of Proposition 3 induces a constant
$C_{3}>0$ (depending on $Q_{1},Q_{2},Q,\mu,k,\nu$) and a constant $\eta_{2}>0$ (equal to $\eta_{2,l}$ from (\ref{low_bds_Htaum_first}))
for which
\begin{multline}
|| \frac{\tau^k}{H(\tau,m)} \int_{0}^{\tau^{k}} \int_{-\infty}^{+\infty}
Q_{1}(i(m-m_{1}))w((\tau^{k}-s)^{1/k},m-m_{1})\\
 \times  Q_{2}(im_{1})
w(s^{1/k},m_{1}) \frac{1}{(\tau^{k}-s)s} dsdm_{1}||_{(\nu,\beta,\mu,k,\rho)}
\leq \sup_{\tau \in S_{d} \cup (D(0,\rho) \setminus L_{-}), m \in \mathbb{R}}
\left| \frac{Q(im)}{H(\tau,m)} \right|\\
\times || \frac{\tau^{k}}{Q(im)} \int_{0}^{\tau^{k}} \int_{-\infty}^{+\infty}
Q_{1}(i(m-m_{1}))w((\tau^{k}-s)^{1/k},m-m_{1})\\
 \times  Q_{2}(im_{1})
w(s^{1/k},m_{1}) \frac{1}{(\tau^{k}-s)s} dsdm_{1}||_{(\nu,\beta,\mu,k,\rho)} \leq
\frac{C_{3}}{\min( \eta_{2} , 1/2)} ||w(\tau,m)||^{2}_{(\nu,\beta,\mu,k,\rho)}\\
\leq \frac{C_{3} \varpi^{2}}{\min( \eta_{2} , 1/2)}. \label{norm_bds_nonlinear_conv_op}
\end{multline}
Furthermore, owing to Lemma 1 and in view of the lower estimates (\ref{low_bds_Htaum_first}), (\ref{low_bds_Htaum_on_cdisc}), we obtain
a constant $K_{f}>0$ (depending on $k,\nu$ and $K_{0},T_{0}$ from (\ref{norm_beta_mu_F_n_a_ln})) and $\eta_{2}>0$ such that
\begin{equation}
|| \frac{\psi(\tau,m,\epsilon)}{H(\tau,m)} ||_{(\nu,\beta,\mu,k,\rho)} \leq \frac{K_{f}}{\min( \eta_{2}, 1/2)
\min_{m \in \mathbb{R}} |Q(im)|} \label{norm_bds_forcing_term}
\end{equation}
for all $\epsilon \in D(0,\epsilon_{0})$.

Now, we select $|c_{12}|,|c_{f}|$ with $|c_{l}|$, $1 \leq l \leq D-1$ small enough in a way that one can find a constant $\varpi>0$ with
\begin{multline}
\sum_{l=1}^{D-1} \epsilon_{0}^{\Delta_{l} - d_{l} + \delta_{l}} |c_{l}| \sum_{n \in I_{l}} \sup_{\epsilon \in D(0,\epsilon_{0})}
|A_{l,n}(\epsilon)| \left( \frac{C_{5} \varpi k^{\delta_l}}{\Gamma( \frac{n + d_{l,k}}{k} )} +
\sum_{1 \leq p \leq \delta_{l}-1} |A_{\delta_{l},p}| \frac{C_{7} \varpi k^{p}}{\Gamma( \frac{n + d_{l,k}}{k} + \delta_{l} - p)} \right)\\
+ |c_{12}| \frac{C_{3} \varpi^{2}}{(2\pi)^{1/2} \min( \eta_{2}, 1/2)}
+ |c_{f}| \frac{K_{f}}{\min( \eta_{2}, 1/2) \min_{m \in \mathbb{R}} |Q(im)|} \leq \varpi. \label{cond_incl_Hepsilon}
\end{multline}
Finally, if one collects the norms estimates (\ref{norm_bds_conv_exp_op_1}), (\ref{norm_bds_conv_exp_op_2}) in a row with 
(\ref{norm_bds_nonlinear_conv_op}) and (\ref{norm_bds_forcing_term}) under the restriction (\ref{cond_incl_Hepsilon}), on gets the
inclusion (\ref{H_epsilon_inclusion}).\medskip

In the next part of the proof, we turn to the second feature (\ref{H_epsilon_shrink}). Namely, let $w_{1}(\tau,m),w_{2}(\tau,m)$ belonging
to $\bar{B}(0,\varpi)$ inside $F_{(\nu,\beta,\mu,k,\rho)}^{d}$. From the very definition, we get in particular that the next bounds
$$
|w_{2}(\tau,m) - w_{1}(\tau,m)| \leq ||w_{2}(\tau,m) - w_{1}(\tau,m)||_{(\nu,\beta,\mu,k,\rho)}
|\tau|^{k} \exp( \nu |\tau|^{k} ) e^{-\beta |m|} (1 + |m|)^{-\mu}
$$
hold for all $\tau \in S_{d} \cup (D(0,\rho) \setminus L_{-})$. Following exactly the same steps as the sequence of inequalities
(\ref{bds_w}), (\ref{bds_exp_kappaCk_Sd_w}), (\ref{bds_op1_sector_inclusion_Hepsilon}),
(\ref{bds_exp_kappaCk_origin_w}), (\ref{bds_op1_cdisc_inclusion_Hepsilon}) and
(\ref{norm_bds_conv_exp_op_1}), we observe that
\begin{multline}
|| \frac{R_{l}(im) \tau^{k}}{H(\tau,m)} \int_{0}^{\tau^k} (\tau^{k} - s)^{\frac{n + d_{l,k}}{k} - 1}
s^{\delta_{l}} (\exp( -\kappa_{l} \mathcal{C})(w_{2} - w_{1}))(s^{1/k},m) \frac{ds}{s} ||_{(\nu,\beta,\mu,k,\rho)}\\
\leq C_{5}||w_{2}(\tau,m) - w_{1}(\tau,m)||_{(\nu,\beta,\mu,k,\rho)}
\label{norm_bds_conv_exp_op_1_shrink}
\end{multline}
for the constant $C_{5}>0$ appearing in (\ref{norm_bds_conv_exp_op_1}).

Similarly, tracking the progression (\ref{bds_w}), (\ref{bds_exp_kappaCk_Sd_w}), (\ref{bds_exp_kappaCk_origin_w}),
(\ref{bds_op2_sector_inclusion_Hepsilon}), (\ref{bds_op2_cdisc_inclusion_Hepsilon}) and (\ref{norm_bds_conv_exp_op_2})
yields the next bounds
\begin{multline}
|| \frac{R_{l}(im) \tau^{k}}{H(\tau,m)} \int_{0}^{\tau^{k}} (\tau^{k} - s)^{\frac{n + d_{l,k}}{k} + \delta_{l} -p -1 } s^{p}
(\exp( - \kappa_{l} \mathcal{C}_{k} )(w_{2}-w_{1}))(s^{1/k},m) \frac{ds}{s} ||_{(\nu,\beta,\mu,k,\rho)}\\
\leq C_{7} ||w_{2}(\tau,m) - w_{1}(\tau,m)||_{(\nu,\beta,\mu,k,\rho)} \label{norm_bds_conv_exp_op_2_shrink}
\end{multline}
for all $1 \leq p \leq \delta_{l}-1$, where the constant $C_{7}>0$ shows up in (\ref{norm_bds_conv_exp_op_2}).

In order to handle the nonlinear term, we need to present the next difference in prepared form
\begin{multline*}
Q_{1}(i(m-m_{1}))w_{2}( (\tau^{k}-s)^{1/k},m-m_{1}) Q_{2}(im_{1})w_{2}(s^{1/k},m_{1})\\
- Q_{1}(i(m-m_{1}))w_{1}((\tau^{k} - s)^{1/k},m-m_{1})Q_{2}(im_{1})w_{1}(s^{1/k},m_{1}) \\
=
Q_{1}(i(m-m_{1})) \left( w_{2}((\tau^{k} - s)^{1/k},m-m_{1}) - w_{1}((\tau^{k} - s)^{1/k},m-m_{1}) \right)
Q_{2}(im_{1})w_{2}(s^{1/k},m_{1})\\
+ Q_{1}(i(m-m_{1}))w_{1}((\tau^{k} - s)^{1/k},m-m_{1})Q_{2}(im_{1}) \left( w_{2}(s^{1/k},m_{1}) - w_{1}(s^{1/k},m_{1}) \right)
\end{multline*}
for all $\tau \in S_{d} \cup (D(0,\rho) \setminus L_{-})$, all $m,m_{1} \in \mathbb{R}$.

Then, in view of the assumption (\ref{constraints_degree_coeff}) and the lower bounds
(\ref{low_bds_Htaum_first}), (\ref{low_bds_Htaum_on_cdisc}), Proposition 3 gives raise to constants
$C_{3}>0$ and $\eta_{2}>0$ appearing in (\ref{norm_bds_nonlinear_conv_op}) for which
\begin{multline}
|| \frac{\tau^{k}}{H(\tau,m)} \int_{0}^{\tau^k} \int_{-\infty}^{+\infty}
\left \{ Q_{1}(i(m-m_{1}))w_{2}((\tau^{k} - s)^{1/k},m-m_{1})Q_{2}(im)w_{2}(s^{1/k},m_{1}) \right. \\
\left. -Q_{1}(i(m-m_{1}))w_{1}((\tau^{k}-s)^{1/k},m-m_{1})Q_{2}(im_{1})w_{1}(s^{1/k},m_{1}) \right \} \frac{1}{(\tau^{k}-s)s} ds dm_{1}
||_{(\nu,\beta,\mu,k,\rho)}\\
\leq \frac{C_{3}}{\min( \eta_{2}, 1/2)} \left( ||w_{2}(\tau,m)||_{(\nu,\beta,\mu,k,\rho)}
+ ||w_{1}(\tau,m)||_{(\nu,\beta,\mu,k,\rho)} \right) ||w_{2}(\tau,m) - w_{1}(\tau,m)||_{(\nu,\beta,\mu,k,\rho)}\\
\leq \frac{2 \varpi C_{3}}{\min( \eta_{2}, 1/2)} ||w_{2}(\tau,m) - w_{1}(\tau,m)||_{(\nu,\beta,\mu,k,\rho)}
\label{norm_bds_nonlinear_conv_op_shrink}
\end{multline}
Now, we restrict the constants $|c_{12}|$ and $|c_{l}|$, $1 \leq l \leq D-1$ in a way that one arrives at the next bounds
\begin{multline}
\sum_{l=1}^{D-1} \epsilon_{0}^{\Delta_{l} - d_{l} + \delta_{l}} |c_{l}| \sum_{n \in I_{l}} \sup_{\epsilon \in D(0,\epsilon_{0})}
|A_{l,n}(\epsilon)| \left( \frac{C_{5} k^{\delta_l}}{\Gamma( \frac{n + d_{l,k}}{k} )} +
\sum_{1 \leq p \leq \delta_{l}-1} |A_{\delta_{l},p}| \frac{C_{7} k^{p}}{\Gamma( \frac{n + d_{l,k}}{k} + \delta_{l} - p)} \right)\\
+ |c_{12}| \frac{2 \varpi C_{3}}{(2\pi)^{1/2} \min( \eta_{2}, 1/2)} \leq \frac{1}{2}. \label{cond_shrink_Hepsilon}
\end{multline}
At last, by assembling the estimates (\ref{norm_bds_conv_exp_op_1_shrink}), (\ref{norm_bds_conv_exp_op_2_shrink}) with
(\ref{norm_bds_nonlinear_conv_op_shrink}) submitted to the constraints (\ref{cond_shrink_Hepsilon}), one achieves the forcast
shrinking property (\ref{H_epsilon_shrink}).

Ultimately, we select $|c_{12}|$, $|c_{f}|$ and
$|c_{l}|$, $1 \leq l \leq D-1$ small enough in a way that (\ref{cond_incl_Hepsilon}) and (\ref{cond_shrink_Hepsilon}) are
simultaneously fulfilled. Lemma 3 follows.
\end{proof}
We turn back again to the proof of Proposition 4. For $\varpi>0$ chosen as in Lemma 3, we set the closed ball
$\bar{B}(0,\varpi) \subset F_{(\nu,\beta,\mu,k,\rho)}^{d}$ which represents a complete metric space for the distance
$d(x,y) = ||x - y||_{(\nu,\beta,\mu,k,\rho)}$. Owing to the lemma above, we observe that $\mathcal{H}_{\epsilon}$ induces a contractive
application from $(\bar{B}(0,\varpi),d)$ into itself. Then, according to the classical contractive mapping theorem, the map
$\mathcal{H}_{\epsilon}$ possesses a unique fixed point that we set as $w^{d}(\tau,m,\epsilon)$, meaning that
\begin{equation}
\mathcal{H}_{\epsilon}(w^{d}(\tau,m,\epsilon) ) = w^{d}(\tau,m,\epsilon), \label{fixed_pt_Hepsilon}
\end{equation}
that belongs to the ball $\bar{B}(0,\varpi)$, for all $\epsilon \in D(0,\epsilon_{0})$. Besides, the function $w^{d}(\tau,m,\epsilon)$
depends holomorphically on $\epsilon$ in $D(0,\epsilon_{0})$. Now, if one moves apart the term
$\exp( \alpha_{D} k \tau^{k} )R_{D}(im)w(\tau,m,\epsilon)$ from the right to the left handside of (\ref{main_integral_eq_w}), we observe
by dividing with the function $H(\tau,m)$ given in (\ref{defin_Htaum}), that (\ref{main_integral_eq_w}) can be exactly rewritten as
the equation (\ref{fixed_pt_Hepsilon}) above. As a result, the unique fixed point $w^{d}(\tau,m,\epsilon)$ of
$\mathcal{H}_{\epsilon}$ obtained overhead in $\bar{B}(0,\varpi)$ precisely solves the equation (\ref{main_integral_eq_w}).
\end{proof}

\section{Analytic solutions on sectors to the main initial value problem}

We turn back to the formal constructions realized in Section 3 by taking into consideration the solution of the related problem
(\ref{main_integral_eq_w}) built up in Section 5 within the Banach spaces described in Definition 4.

At the onset, we remind the reader the definition of a good covering in $\mathbb{C}^{\ast}$ and we disclose a modified version of
so-called associated sets of sectors as proposed in our previous work \cite{lama1}.

\begin{defin} Let $\varsigma \geq 2$ be an integer. For all
$0 \leq p \leq \varsigma - 1$, we set $\mathcal{E}_{p}$ as an open sector centered at 0, with radius $\epsilon_{0}>0$
such that $\mathcal{E}_{p} \cap \mathcal{E}_{p+1} \neq \emptyset$ for all
$0 \leq p \leq \varsigma-1$ (with the convention that $\mathcal{E}_{\varsigma} = \mathcal{E}_{0}$). Furthermore, we take for granted that the
intersection of any three different elements of $\{ \mathcal{E}_{p} \}_{0 \leq p \leq \varsigma - 1}$ is empty and that
$\cup_{p=0}^{\varsigma - 1} \mathcal{E}_{p} = \mathcal{U} \setminus \{ 0 \}$, where $\mathcal{U}$ stands for some neighborhood
of 0 in $\mathbb{C}$. A set of sector $\{ \mathcal{E}_{p} \}_{0 \leq p \leq \varsigma - 1}$ with the above properties is called a good covering
in $\mathbb{C}^{\ast}$.
\end{defin}

\begin{defin}
We consider a good covering $\underline{\mathcal{E}} = \{ \mathcal{E}_{p} \}_{0 \leq p \leq \varsigma - 1}$ in $\mathbb{C}^{\ast}$. We fix a
real number $\rho>0$
and an open sector $\mathcal{T}$ centered at 0 with bisecting direction $d=0$ and radius $r_{\mathcal{T}}>0$ and we set up a family of open
sectors
$$ S_{\mathfrak{d}_{p},\theta,\epsilon_{0}r_{\mathcal{T}}} = \{ T \in \mathbb{C}^{\ast}/ \ \ |T| < \epsilon_{0}r_{\mathcal{T}},
\ \ | \mathfrak{d}_{p} - \mathrm{arg}(T)| < \theta/2 \} $$
with aperture $\theta > \pi/k$ and $\mathfrak{d}_{p} \in [-\pi,\pi)$, $0 \leq p \leq \varsigma - 1$ represent their bisecting directions.
We say that the data $\{ \{S_{\mathfrak{d}_{p},\theta,\epsilon_{0}r_{\mathcal{T}}} \}_{0 \leq p \leq \varsigma -1},
\mathcal{T}, \rho \}$ are associated to $\underline{\mathcal{E}}$ if the next two constraints hold:\\
1) There exists a set of unbounded sectors $S_{\mathfrak{d}_{p}}$, $0 \leq p \leq \varsigma-1$ centered at 0 with suitably chosen
bisecting direction $\mathfrak{d}_{p} \in (- \pi/2,\pi/2 )$ and small aperture satisfying the property that
$$ \tau_{l} \notin S_{\mathfrak{d}_{p}} \cup D(0,\rho) $$
for some fixed radius $\rho>0$ and all $l \in \mathbb{Z}$ where $\tau_{l}$ stand for the complex numbers defined through
(\ref{defin_taul}).\\
2) For all $\epsilon \in \mathcal{E}_{p}$, all $t \in \mathcal{T}$,
\begin{equation}
\epsilon t \in S_{\mathfrak{d}_{p},\theta,\epsilon_{0}r_{\mathcal{T}}} \label{epsilon_t_in_sector}
\end{equation}
for all $0 \leq p \leq \varsigma - 1$.
\end{defin}

Figure~\ref{fig:c} shows a configuration of a good covering of three sectors, one of them of opening larger than $\pi/k$ for some $k$ close to 1. We illustrate in Figure~\ref{fig:d} a configuration of associated sectors.

\begin{figure}
	\centering
		\includegraphics[width=0.4\textwidth]{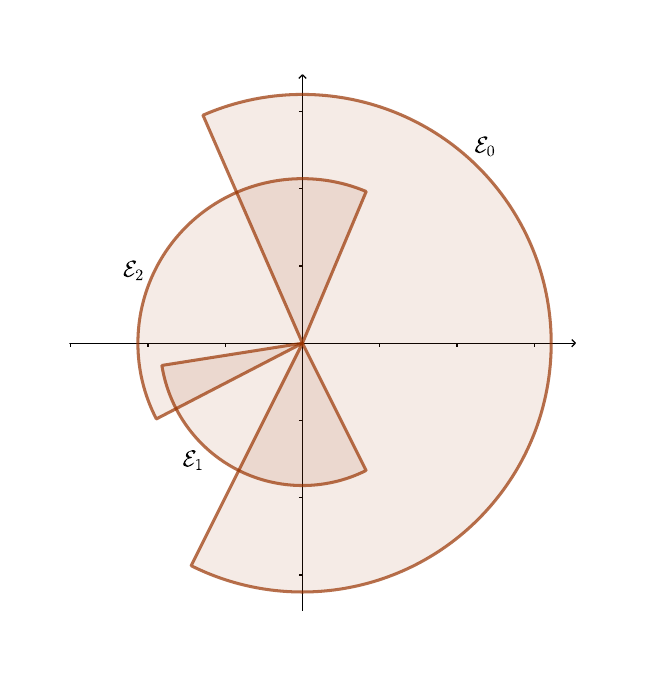}
				\caption{Good covering in $\mathbb{C}^\star$}
				\label{fig:c}
\end{figure}

\begin{figure}
	\centering
	\includegraphics[width=0.3\textwidth]{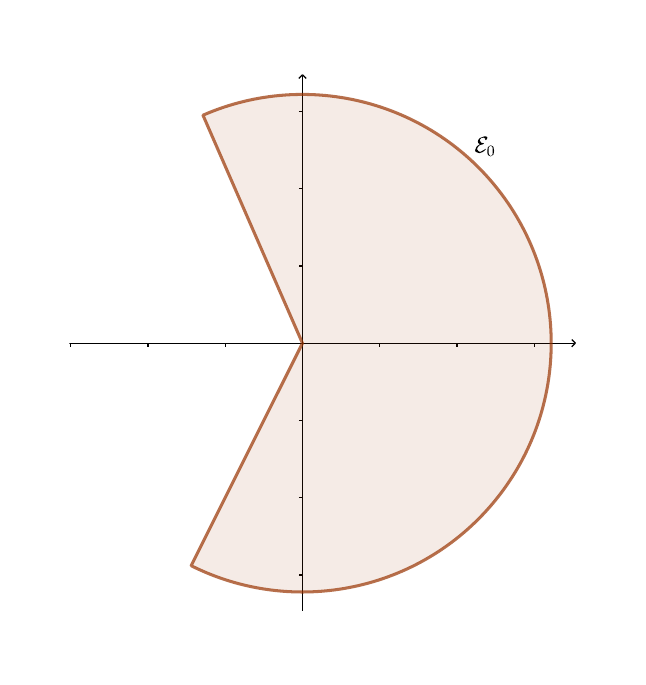} \vline\vline \includegraphics[width=0.3\textwidth]{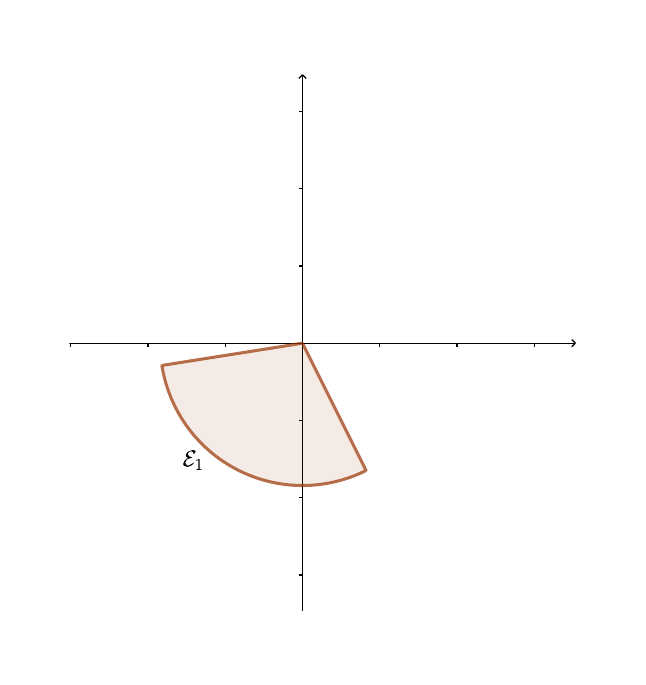} \vline\vline \includegraphics[width=0.3\textwidth]{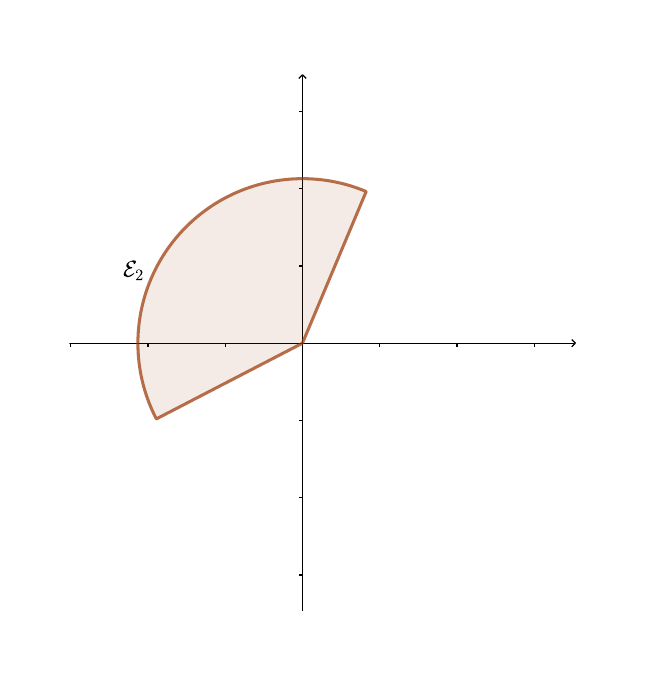}\\
	\includegraphics[width=0.3\textwidth]{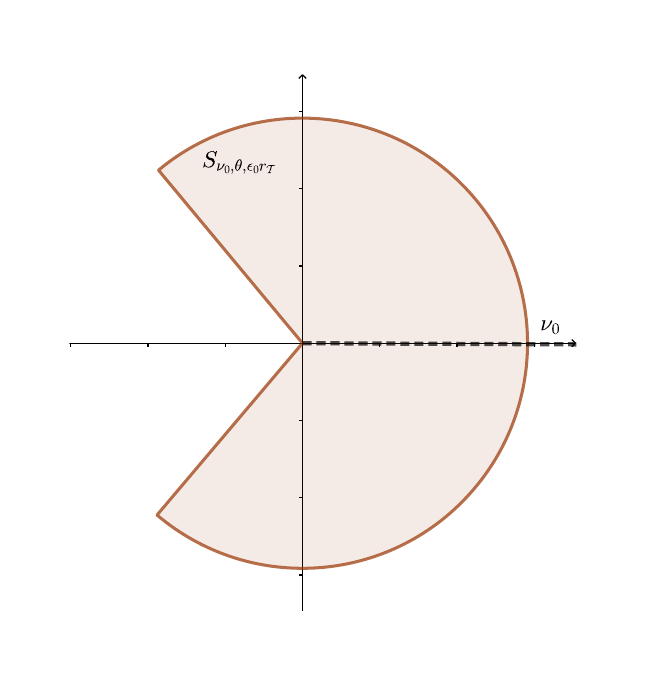} \vline\vline \includegraphics[width=0.3\textwidth]{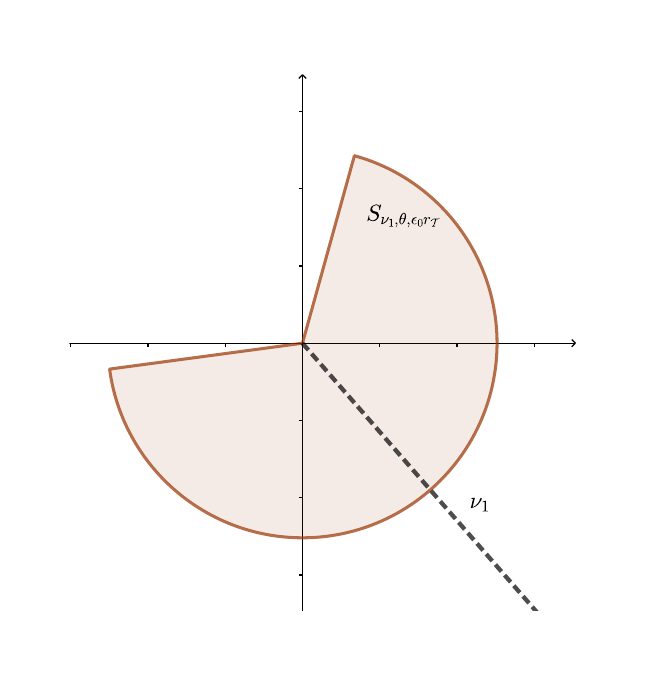} \vline\vline \includegraphics[width=0.3\textwidth]{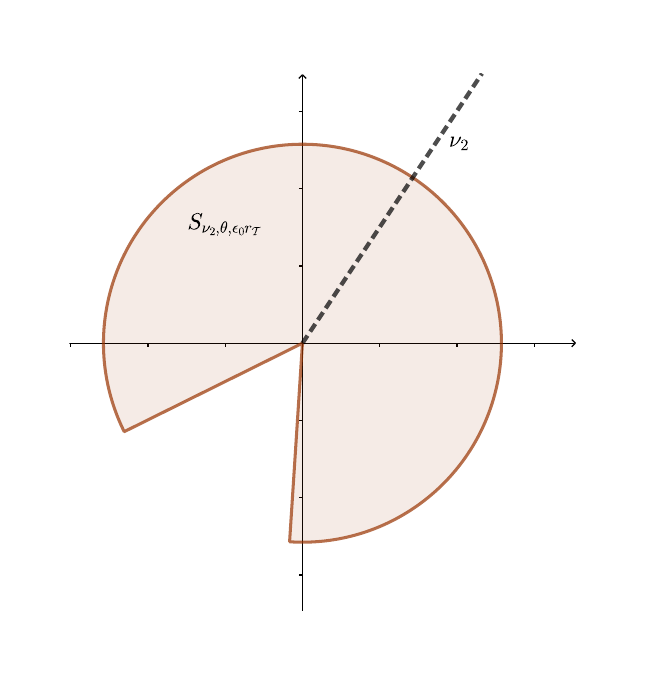}
	\caption{A configuration associated to the good covering in Figure~\ref{fig:c}}
	\label{fig:d}
\end{figure}

In the following first principal result of the work, we build up a set of actual holomorphic solutions to the main initial value problem
(\ref{main_ivp_u}) defined on the sectors $\mathcal{E}_{p}$ w.r.t $\epsilon$. We also provide an upper control for the difference between any two
neighboring solutions on $\mathcal{E}_{p} \cap \mathcal{E}_{p+1}$ that turn out to be at most exponentially flat of order $k$.

\begin{theo}
Let us assume that the constraints (\ref{defin_deltal}), (\ref{cond_dl_deltal_Deltal}), (\ref{constraints_degree_coeff}),
(\ref{norm_beta_mu_F_n_a_ln}) and (\ref{quotient_Q_RD_in_SQRD}) hold. We consider a good covering
$\underline{\mathcal{E}} = \{ \mathcal{E}_{p} \}_{0 \leq p \leq \varsigma - 1}$ for which a set of data
$\{ \{S_{\mathfrak{d}_{p},\theta,\epsilon_{0}r_{\mathcal{T}}} \}_{0 \leq p \leq \varsigma - 1}, \mathcal{T}, \rho \}$ associated to
$\underline{\mathcal{E}}$ can be singled out. We take for granted that the constants $\alpha_{D}$ and
$\kappa_{l}$, $1 \leq l \leq D-1$ appearing in the problem (\ref{main_ivp_u}) are submitted to the next inequalities
\begin{equation}
B_{H,\mathfrak{d}_{p}} \alpha_{D} > \kappa_{l} K_{1} \frac{k}{\Gamma(\frac{1}{k} - 1)}
\end{equation}
for all $0 \leq p \leq \varsigma - 1$, where $B_{H,\mathfrak{d}_{p}}$ is framed in the construction
(\ref{low_bds_Htaum_on_sector}) and depends on $k$,$S_{Q,R_{D}}$, $S_{\mathfrak{d}_{p}}$ and $K_{1}>0$ is a constant
relying on $k,\nu$ defined in Proposition 2 1).

Then, whenever the moduli $|c_{12}|$,$|c_{f}|$ and $|c_{l}|$, $1 \leq l \leq D-1$ are taken sufficiently small, a family
$\{u_{p}(t,z,\epsilon)\}_{0 \leq p \leq \varsigma - 1}$ of genuine solutions of (\ref{main_ivp_u}) can be established.
More precisely, each function $u_{p}(t,z,\epsilon)$ defines a bounded holomorphic function on the product
$(\mathcal{T} \cap D(0,\sigma)) \times H_{\beta'} \times \mathcal{E}_{p}$ for any given $0 < \beta' < \beta$ and suitably tiny $\sigma>0$
(where $\beta$ comes out in (\ref{norm_beta_mu_F_n_a_ln})) and can be expressed as a Laplace transform of order $k$ and
Fourier inverse transform
\begin{equation}
 u_{p}(t,z,\epsilon) = \frac{k}{(2\pi)^{1/2}} \int_{-\infty}^{+\infty} \int_{L_{\gamma_{p}}}
 w^{\mathfrak{d}_{p}}(u,m,\epsilon) \exp( -(\frac{u}{\epsilon t})^{k} ) e^{izm} \frac{du}{u} dm \label{Laplace_Fourier_up}
\end{equation}
along a halfline $L_{\gamma_{p}} = \mathbb{R}_{+}e^{\sqrt{-1}\gamma_{p}} \subset S_{\mathfrak{d}_{p}} \cup \{ 0 \}$ and where
$w^{\mathfrak{d}_{p}}(\tau,m,\epsilon)$ stands for a function that belongs to the Banach space
$F^{\mathfrak{d}_{p}}_{(\nu,\beta,\mu,k,\rho)}$ for all $\epsilon \in D(0,\epsilon_{0})$. Furthermore, one can choose constants
$K_{p},M_{p}>0$ and $0< \sigma' < \sigma$ (independent of $\epsilon$) with
\begin{equation}
\sup_{t \in \mathcal{T} \cap D(0,\sigma'),z \in H_{\beta'}} |u_{p+1}(t,z,\epsilon) -
u_{p}(t,z,\epsilon)| \leq K_{p} \exp( -\frac{M_{p}}{|\epsilon|^{k}} ) \label{exp_small_difference_u_p}
\end{equation}
for all $\epsilon \in \mathcal{E}_{p+1} \cap \mathcal{E}_{p}$, all $0 \leq p \leq \varsigma-1$ (owing to the convention that
$u_{\varsigma} = u_{0}$).
\end{theo}
\begin{proof}
Our goal will be to shape true solutions of the main equation (\ref{main_ivp_u}) by performing rearward the sequence of constructions disclosed in
Section 3 departing from the problem (\ref{main_integral_eq_w}) that has been worked out in Section 5.

We select a good covering $\{ \mathcal{E}_{p} \}_{0 \leq p \leq \varsigma - 1}$ in $\mathbb{C}^{\ast}$ for which a set of data
$\{ \{S_{\mathfrak{d}_{p},\theta,\epsilon_{0}r_{\mathcal{T}}} \}_{0 \leq p \leq \varsigma - 1}, \mathcal{T}, \rho \}$ can be associated.
According to Proposition 4, under the assumptions stated in Theorem 1, for suitably chosen moduli $|c_{12}|$, $|c_{f}|$ and
$|c_{l}|$, $1 \leq l \leq D-1$, we observe that for each direction $\mathfrak{d}_{p}$, one can build a solution
$w^{\mathfrak{d}_{p}}(\tau,m,\epsilon)$ of the convolution equation (\ref{main_integral_eq_w}) within the space
$F^{\mathfrak{d}_{p}}_{(\nu,\beta,\mu,k,\rho)}$ that fulfills the next estimates
\begin{equation}
|w^{\mathfrak{d}_{p}}(\tau,m,\epsilon)| \leq \varpi_{\mathfrak{d}_{p}} (1 + |m|)^{-\mu} e^{-\beta |m|}
\frac{|\tau|^{k}}{1 + |\tau|^{2k}} \exp( \nu |\tau|^{k} ) \label{bds_w_dp_varpi}
\end{equation}
for all $\tau \in S_{\mathfrak{d}_{p}} \cup (D(0,\rho) \setminus L_{-})$, all $m \in \mathbb{R}$, all $\epsilon \in D(0,\epsilon_{0})$, for
well chosen $\varpi_{\mathfrak{d}_{p}}>0$. For later need, we see in particular that $w^{\mathfrak{d}_{p}}(\tau,m,\epsilon)$ are analytic
continuation w.r.t $\tau$
of a common holomorphic function that we call $\tau \mapsto w(\tau,m,\epsilon)$ whenever $\tau \in D(0,\rho) \setminus L_{-}$ which
satisfies likewise the bounds above (\ref{bds_w_dp_varpi}) provided that $m \in \mathbb{R}$, $\epsilon \in D(0,\epsilon_{0})$.

As a consequence, the Laplace transform of order $k$ and Fourier inverse transform
$$ U_{\gamma_{p}}(T,z,\epsilon) = \frac{k}{(2\pi)^{1/2}} \int_{-\infty}^{+\infty}
\int_{L_{\gamma_{p}}} w^{\mathfrak{d}_{p}}(u,m,\epsilon)
\exp( -(\frac{u}{T})^{k} ) e^{izm} \frac{du}{u} dm $$
along a halfline $L_{\gamma_{p}} \subset S_{\mathfrak{d}_{p}} \cup \{ 0 \}$
represents\\
1) A holomorphic bounded function w.r.t $T$ on a sector $S_{\mathfrak{d}_{p},\theta,\varrho}$ with bisecting direction $\mathfrak{d}_{p}$,
aperture $\frac{\pi}{k} < \theta < \frac{\pi}{k} + \mathrm{Ap}(S_{\mathfrak{d}_p})$, radius $\varrho$, where
$\mathrm{Ap}(S_{\mathfrak{d}_p})$ stands for the aperture of $S_{\mathfrak{d}_{p}}$, for some real number $\varrho>0$.\\
2) A holomorphic bounded application w.r.t $z$ on $H_{\beta'}$ for any given $0 < \beta' < \beta$.\\
3) A holomorphic bounded map w.r.t $\epsilon$ on $D(0,\epsilon_{0})$.

Furthermore, the integral representation (\ref{Moebius_Ugamma_int_rep}) accompanied with the bounds (\ref{bds_exp_kappaCk_Sd}),
(\ref{bds_exp_kappaCk_origin}) shows that $U_{\gamma_{p}}(\frac{T}{1 + \kappa_{l}T},z,\epsilon)$ defines a holomorphic bounded function
w.r.t $T$ on a sector $S_{\mathfrak{d}_{p},\theta,\varrho_{1}}$ for some $0 < \varrho_{1} < \varrho$ and the same
$\theta$ as above in 1). Besides, a direct computation yields that
\begin{multline}
\exp( \alpha_{D} T^{k+1} \partial_{T} ) U_{\gamma_{p}}(T,z,\epsilon) \\
=
\frac{k}{(2\pi)^{1/2}} \int_{-\infty}^{+\infty} \int_{L_{\gamma_{p}}} \exp( \alpha_{D} ku^{k} ) w^{\mathfrak{d}_{p}}(u,m,\epsilon)
\exp( -(\frac{u}{T})^{k} ) e^{izm} \frac{du}{u} dm \label{expTk_on_Udp}
\end{multline}
prescribes a holomorphic bounded function w.r.t $T$ on a sector $S_{\mathfrak{d}_{p},\theta,\varrho_{2}}$ for some $0 < \varrho_{2} < \varrho$
for the overhead value of $\theta$.

Finally, by combining these integral representations (\ref{Moebius_Ugamma_int_rep}), (\ref{expTk_on_Udp})
together with those (\ref{TkpartialTUgamma}), (\ref{TmUgamma}) and (\ref{QUgammaQUgamma}) displayed in Lemma 2, from
the fact that $w^{\mathfrak{d}_{p}}$ solves (\ref{main_integral_eq_w}), we obtain that $U_{\gamma_{p}}(T,z,\epsilon)$ solves the equation
(\ref{main_ivp_U_prep_form}) and hence the equation (\ref{main_ivp_U}) whenever $T$ belongs to some sector
$S_{\mathfrak{d}_{p},\theta,\varrho_{3}}$ where $0 < \varrho_{3} < \varrho$, if $z$ lies within $H_{\beta'}$ and $\epsilon \in D(0,\epsilon_{0})$.

As a result, the function
$$ u_{p}(t,z,\epsilon) = U_{\gamma_{p}}(\epsilon t,z,\epsilon) $$
defines a bounded holomorphic function w.r.t $t$ on $\mathcal{T} \cap D(0,\sigma)$ for some $\sigma>0$ small enough,
$\epsilon \in \mathcal{E}_{p}$, $z \in H_{\beta'}$ for any given $0 < \beta' < \beta$, owing to the fact that the sectors $\mathcal{E}_{p}$
and $\mathcal{T}$ from the associated data fulfill the crucial feature (\ref{epsilon_t_in_sector}). Moreover,
$u_{p}(t,z,\epsilon)$ solves the main initial value problem (\ref{main_ivp_u}) on the domain described above
$(\mathcal{T} \cap D(0,\sigma)) \times H_{\beta'} \times \mathcal{E}_{p}$, for all $0 \leq p \leq \varsigma - 1$.

In the final part of the proof, we are concerned with the bounds (\ref{exp_small_difference_u_p}). The steps of verification are comparable
to the arguments displayed in Theorem 1 of \cite{lama1} but we still decide to present the details for the benefit of clarity.

By construction, the map $u \mapsto w(u,m,\epsilon) \exp( -(\frac{u}{\epsilon t})^{k} )/u$ represents a
holomorphic function on $D(0,\rho) \setminus L_{-}$ for all
$(m,\epsilon) \in \mathbb{R} \times D(0,\epsilon_{0})$. Therefore, its integral along the union of a segment joining
0 to $(\rho/2)e^{\sqrt{-1}\gamma_{p+1}}$ followed by an arc of circle with radius $\rho/2$ which relies
$(\rho/2)e^{\sqrt{-1}\gamma_{p+1}}$ and $(\rho/2)e^{\sqrt{-1}\gamma_{p}}$ and ending with a segment starting from
$(\rho/2)e^{\sqrt{-1}\gamma_{p}}$ to 0, is vanishing. The Cauchy formula allows us to write the difference $u_{p+1} - u_{p}$ as a sum of three
integrals,
\begin{multline}
u_{p+1}(t,z,\epsilon) - u_{p}(t,z,\epsilon) = \frac{k}{(2\pi)^{1/2}}\int_{-\infty}^{+\infty}
\int_{L_{\rho/2,\gamma_{p+1}}}
w^{\mathfrak{d}_{p+1}}(u,m,\epsilon) e^{-(\frac{u}{\epsilon t})^{k}} e^{izm} \frac{du}{u} dm\\ -
\frac{k}{(2\pi)^{1/2}}\int_{-\infty}^{+\infty}
\int_{L_{\rho/2,\gamma_{p}}}
w^{\mathfrak{d}_p}(u,m,\epsilon) e^{-(\frac{u}{\epsilon t})^{k}} e^{izm} \frac{du}{u} dm\\
+ \frac{k}{(2\pi)^{1/2}}\int_{-\infty}^{+\infty}
\int_{C_{\rho/2,\gamma_{p},\gamma_{p+1}}}
w(u,m,\epsilon) e^{-(\frac{u}{\epsilon t})^{k}} e^{izm} \frac{du}{u} dm \label{difference_u_p_decomposition}
\end{multline}
where $L_{\rho/2,\gamma_{p+1}} = [\rho/2,+\infty)e^{\sqrt{-1}\gamma_{p+1}}$,
$L_{\rho/2,\gamma_{p}} = [\rho/2,+\infty)e^{\sqrt{-1}\gamma_{p}}$ and
$C_{\rho/2,\gamma_{p},\gamma_{p+1}}$ stands for an arc of circle with radius connecting
$(\rho/2)e^{\sqrt{-1}\gamma_{p}}$ and $(\rho/2)e^{\sqrt{-1}\gamma_{p+1}}$ with a well chosen orientation.\medskip

We first provide bounds for the front part of the decomposition (\ref{difference_u_p_decomposition}), namely
$$ I_{1} = \left| \frac{k}{(2\pi)^{1/2}}\int_{-\infty}^{+\infty}
\int_{L_{\rho/2,\gamma_{p+1}}}
w^{\mathfrak{d}_{p+1}}(u,m,\epsilon) e^{-(\frac{u}{\epsilon t})^{k}} e^{izm} \frac{du}{u} dm \right|.
$$
By construction, the direction $\gamma_{p+1}$ (which may depend on $\epsilon t$) is chosen in such a way that
$\cos( k( \gamma_{p+1} - \mathrm{arg}(\epsilon t) )) \geq \delta_{1}$, for all
$\epsilon \in \mathcal{E}_{p} \cap \mathcal{E}_{p+1}$, all $t \in \mathcal{T} \cap D(0,\sigma)$, for some fixed $\delta_{1} > 0$.
From the estimates (\ref{bds_w_dp_varpi}), we get that
\begin{multline}
I_{1} \leq \frac{k}{(2\pi)^{1/2}} \int_{-\infty}^{+\infty} \int_{\rho/2}^{+\infty}
\varpi_{\mathfrak{d}_{p+1}}(1+|m|)^{-\mu} e^{-\beta|m|}
\frac{r^{k}}{1 + r^{2k} } \\
\times \exp( \nu r^{k} )
\exp(-\frac{\cos(k(\gamma_{p+1} - \mathrm{arg}(\epsilon t)))}{|\epsilon t|^{k}}r^{k}) e^{-m\mathrm{Im}(z)} \frac{dr}{r} dm\\
\leq \frac{k\varpi_{\mathfrak{d}_{p+1}}}{(2\pi)^{1/2}} \int_{-\infty}^{+\infty} e^{-(\beta - \beta')|m|} dm
\int_{\rho/2}^{+\infty} r^{k-1} \exp( -(\frac{\delta_{1}}{|t|^{k}} - \nu |\epsilon|^{k})(\frac{r}{|\epsilon|})^{k} ) dr\\
\leq  \frac{2k\varpi_{\mathfrak{d}_{p+1}}}{(2\pi)^{1/2}} \int_{0}^{+\infty} e^{-(\beta - \beta')m} dm
\int_{\rho/2}^{+\infty} \frac{|\epsilon|^{k}}{(\frac{\delta_{1}}{|t|^{k}} - \nu |\epsilon|^{k}) k}
\times\\
\left\{ \frac{(\frac{\delta_{1}}{|t|^{k}} - \nu |\epsilon|^{k}) }{|\epsilon|^{k}} kr^{k-1}
\exp( -(\frac{\delta_{1}}{|t|^{k}} - \nu |\epsilon|^{k})(\frac{r}{|\epsilon|})^{k} ) \right\} dr \leq
\frac{2k\varpi_{\mathfrak{d}_{p+1}}}{(2\pi)^{1/2}} \frac{ |\epsilon|^{k} }{(\beta - \beta')( \frac{\delta_{1}}{|t|^{k}} - \nu |\epsilon|^{k})k}\\
\times
\exp( - (\frac{\delta_{1}}{|t|^{k}} - \nu |\epsilon|^{k}) (\frac{\rho/2}{|\epsilon|})^{k} ) \leq 
\frac{2k\varpi_{\mathfrak{d}_{p+1}}}{(2\pi)^{1/2}} \frac{|\epsilon|^{k}}{(\beta - \beta') \delta_{2}k }
\exp( -\delta_{2} (\frac{\rho/2}{|\epsilon|})^{k} ) \label{I_1_exp_small_order_k}
\end{multline}
for all $t \in \mathcal{T} \cap D(0,\sigma)$ and $z \in H_{\beta'}$ with
$|t| < (\frac{\delta_{1}}{\delta_{2} + \nu \epsilon_{0}^{k}})^{1/k}$, for some $\delta_{2}>0$, whenever
$\epsilon \in \mathcal{E}_{p} \cap \mathcal{E}_{p+1}$.\medskip

In a similar manner, we supply estimates for the middle part of (\ref{difference_u_p_decomposition}), especially
$$ I_{2} = \left| \frac{k}{(2\pi)^{1/2}}\int_{-\infty}^{+\infty}
\int_{L_{\rho/2,\gamma_{p}}}
w^{\mathfrak{d}_{p}}(u,m,\epsilon) e^{-(\frac{u}{\epsilon t})^{k}} e^{izm} \frac{du}{u} dm \right|.
$$
As above, the direction $\gamma_{p}$ (which relies on $\epsilon t$) is taken in a way that
$\cos( k( \gamma_{p} - \mathrm{arg}(\epsilon t) )) \geq \delta_{1}$, for all
$\epsilon \in \mathcal{E}_{p} \cap \mathcal{E}_{p+1}$, all $t \in \mathcal{T} \cap D(0,\sigma)$, for some fixed $\delta_{1} > 0$.
Again, accordingly to (\ref{bds_w_dp_varpi}) we obtain
\begin{equation}
I_{2} \leq \frac{2k\varpi_{\mathfrak{d}_{p}}}{(2\pi)^{1/2}} \frac{|\epsilon|^{k}}{(\beta - \beta')
\delta_{2}k} \exp( -\delta_{2} \frac{(\rho/2)^k}{|\epsilon|^k} ) \label{I_2_exp_small_order_k}
\end{equation}
for all $t \in \mathcal{T} \cap D(0,\sigma)$ and $z \in H_{\beta'}$ with
$|t| < (\frac{\delta_{1}}{\delta_{2} + \nu \epsilon_{0}^{k}})^{1/k}$, for some $\delta_{2}>0$, for all
$\epsilon \in \mathcal{E}_{p} \cap \mathcal{E}_{p+1}$.\medskip

Lastly, we deal with the remaining piece of (\ref{difference_u_p_decomposition}), that is 
$$
I_{3} = \left| \frac{k}{(2\pi)^{1/2}}\int_{-\infty}^{+\infty}
\int_{C_{\rho/2,\gamma_{p},\gamma_{p+1}}}
w(u,m,\epsilon) e^{-(\frac{u}{\epsilon t})^{k}} e^{izm} \frac{du}{u} dm \right|.
$$
By construction, the arc of circle $C_{\rho/2,\gamma_{p},\gamma_{p+1}}$ is chosen appropriately in order that
$\cos(k(\theta - \mathrm{arg}(\epsilon t))) \geq \delta_{1}$, for all $\theta \in [\gamma_{p},\gamma_{p+1}]$ (if
$\gamma_{p} < \gamma_{p+1}$), $\theta \in [\gamma_{p+1},\gamma_{p}]$ (if
$\gamma_{p+1} < \gamma_{p}$), for all $t \in \mathcal{T}$, all $\epsilon \in \mathcal{E}_{p} \cap \mathcal{E}_{p+1}$, for some
fixed $\delta_{1}>0$.\medskip

Owing to (\ref{bds_w_dp_varpi}) we notice that
\begin{multline}
I_{3} \leq \frac{k}{(2\pi)^{1/2}} \int_{-\infty}^{+\infty}  \left| \int_{\gamma_{p}}^{\gamma_{p+1}} \right.
\max_{0 \leq p \leq \varsigma - 1} \varpi_{\mathfrak{d}_p} (1+|m|)^{-\mu} e^{-\beta|m|}
\frac{ (\rho/2)^{k}}{1 + (\rho/2)^{2k} } \exp( \nu (\rho/2)^{k} )\\
\times
\exp(-\frac{\cos(k(\theta - \mathrm{arg}(\epsilon t)))}{|\epsilon t|^{k}}(\frac{\rho}{2})^{k})
\left. e^{-m\mathrm{Im}(z)} d\theta \right| dm \leq \frac{k \max_{0 \leq p \leq \varsigma - 1} \varpi_{\mathfrak{d}_p} }{(2\pi)^{1/2}}
\int_{-\infty}^{+\infty} e^{-(\beta - \beta')|m|} dm \\
\times |\gamma_{p} - \gamma_{p+1}| (\rho/2)^{k}
\exp( - (\frac{\delta_{1}}{|t|^{k}} - \nu |\epsilon|^{k}) (\frac{\rho/2}{|\epsilon|})^{k} ) \\
\leq
\frac{2k \max_{0 \leq p \leq \varsigma - 1} \varpi_{\mathfrak{d}_p} }{(2\pi)^{1/2}}
\frac{|\gamma_{p} - \gamma_{p+1}|(\rho/2)^{k}}{\beta - \beta'} \exp( -\delta_{2} \frac{(\rho/2)^{k}}{|\epsilon|^{k}} )
\label{I_3_exp_small_order_k}
\end{multline}
for all $t \in \mathcal{T} \cap D(0,\sigma)$ and $z \in H_{\beta'}$ whenever
$|t| < (\frac{\delta_{1}}{\delta_{2} + \nu \epsilon_{0}^{k}})^{1/k}$, for some $\delta_{2}>0$, for all
$\epsilon \in \mathcal{E}_{p} \cap \mathcal{E}_{p+1}$.\medskip

Ultimately, by collecting the three above inequalities (\ref{I_1_exp_small_order_k}), (\ref{I_2_exp_small_order_k}) and
(\ref{I_3_exp_small_order_k}), we figure out that
\begin{multline*}
|u_{p+1}(t,z,\epsilon) - u_{p}(t,z,\epsilon)| \\
\leq
( \frac{2k(\varpi_{\mathfrak{d}_{p}} + \varpi_{\mathfrak{d}_{p+1}})}{(2\pi)^{1/2}} \frac{|\epsilon|^{k}}{(\beta - \beta')
\delta_{2}k} + \frac{2k \max_{0 \leq p \leq \varsigma - 1} \varpi_{\mathfrak{d}_p} }{(2\pi)^{1/2}}
\frac{|\gamma_{p} - \gamma_{p+1}|(\rho/2)^{k}}{\beta - \beta'} )\\
\times \exp( -\delta_{2} \frac{(\rho/2)^k}{|\epsilon|^k} )
\end{multline*}
for all $t \in \mathcal{T} \cap D(0,\sigma)$ and $z \in H_{\beta'}$ with
$|t| < (\frac{\delta_{1}}{\delta_{2} + \nu \epsilon_{0}^{k}})^{1/k}$, for some $\delta_{2}>0$, for all
$\epsilon \in \mathcal{E}_{p} \cap \mathcal{E}_{p+1}$. The forecast inequality
(\ref{exp_small_difference_u_p}) follows.
\end{proof}

\section{Parametric Gevrey asymptotic expansions of order $1/k$ of the solutions}

\subsection{Gevrey asymptotic expansions of order $1/k$ and $k-$summable formal series}

We first remind the reader the concept of $k-$summability of formal series with coefficients in a Banach space as defined in
classical textbooks such as \cite{ba}.

\begin{defin} We set $(\mathbb{F},||.||_{\mathbb{F}})$ as a complex Banach space. Let $k \in (\frac{1}{2},1)$ be a real number. A formal series
$$\hat{a}(\epsilon) = \sum_{j=0}^{\infty}  a_{j}  \epsilon^{j} \in \mathbb{F}[[\epsilon]]$$
with coefficients belonging to $( \mathbb{F}, ||.||_{\mathbb{F}} )$ is called $k-$summable
with respect to $\epsilon$ in the direction $d \in \mathbb{R}$ if \medskip

{\bf i)} One can select a radius $\rho \in \mathbb{R}_{+}$ in a way that the formal series, called formal
Borel transform of order $k$ of $\hat{a}$,
$$ B_{k}(\hat{a})(\tau) = \sum_{j=0}^{\infty} \frac{ a_{j} \tau^{j}  }{ \Gamma(1 + \frac{j}{k}) } \in \mathbb{F}[[\tau]],$$
converge absolutely for $|\tau| < \rho$. \medskip

{\bf ii)} One can find an aperture $2\delta > 0$ such that the series $B_{k}(\hat{a})(\tau)$ can be analytically continued with
respect to $\tau$ on the unbounded sector
$S_{d,\delta} = \{ \tau \in \mathbb{C}^{\ast} : |d - \mathrm{arg}(\tau) | < \delta \} $. Moreover, there exist $C >0$ and $K >0$
with
$$ ||B_{k}(\hat{a})(\tau)||_{\mathbb{F}}
\leq C e^{ K|\tau|^{k}} $$
for all $\tau \in S_{d, \delta}$.
\end{defin}
If the conditions above are fulfilled, the vector valued Laplace transform of order $k$ of $B_{k}(\hat{a})(\tau)$
in the direction $d$ is defined by
$$ L^{d}_{k}(B_{k}(\hat{a}))(\epsilon) = \epsilon^{-k} \int_{L_{\gamma}}
B_{k}(\hat{a})(u) e^{ - ( u/\epsilon )^{k} } ku^{k-1}du,$$
along a half-line $L_{\gamma} = \mathbb{R}_{+}e^{\sqrt{-1}\gamma} \subset S_{d,\delta} \cup \{ 0 \}$, where $\gamma$ depends on
$\epsilon$ and is chosen in such a way that
$\cos(k(\gamma - \mathrm{arg}(\epsilon))) \geq \delta_{1} > 0$, for some fixed $\delta_{1}$, for all
$\epsilon$ in a sector
$$ S_{d,\theta,R^{1/k}} = \{ \epsilon \in \mathbb{C}^{\ast} : |\epsilon| < R^{1/k} \ \ , \ \ |d - \mathrm{arg}(\epsilon) |
< \theta/2 \},$$
where the angle $\theta$ and radius $R$ suffer the next restrictions, $\frac{\pi}{k} < \theta < \frac{\pi}{k} + 2\delta$ and $0 < R < \delta_{1}/K$.

Notice that this Laplace transform of
order $k$ differs slightly from the one introduced in Definition 1 which turns out to be more suitable for the problems under study in this work.

The function $L^{d}_{k}(B_{k}(\hat{a}))(\epsilon)$
is called the $k-$sum of the formal series $\hat{a}(\epsilon)$ in the direction $d$. It represents a bounded and holomorphic function on the sector
$S_{d,\theta,R^{1/k}}$ and is the \emph{unique} such function that possesses the formal series $\hat{a}(\epsilon)$ as Gevrey asymptotic
expansion of order $1/k$ with respect to $\epsilon$ on $S_{d,\theta,R^{1/k}}$ which means that for all
$\frac{\pi}{k} < \theta_{1} < \theta$, there exist $C,M > 0$ such that
$$ ||L^{d}_{k}(B_{k}(\hat{a}))(\epsilon) - \sum_{p=0}^{n-1}
a_{p} \epsilon^{p}||_{\mathbb{F}} \leq CM^{n}\Gamma(1+ \frac{n}{k})|\epsilon|^{n} $$
for all $n \geq 1$, all $\epsilon \in S_{d,\theta_{1},R^{1/k}}$.\medskip

In the sequel, we present a cohomological criterion for the existence of Gevrey asymptotics of order $1/k$ for suitable families of sectorial
holomorphic functions and $k-$summability of formal series with coefficients in Banach spaces (see
\cite{ba2}, p. 121 or \cite{hssi}, Lemma XI-2-6) which is known as the Ramis-Sibuya theorem in the literature. This result
is an essential tool in the proof of our second main statement (Theorem 2).\medskip

\noindent {\bf Theorem (RS)} {\it Let $(\mathbb{F},||.||_{\mathbb{F}})$ be a Banach space over $\mathbb{C}$ and
$\{ \mathcal{E}_{p} \}_{0 \leq p \leq \varsigma-1}$ be a good covering in $\mathbb{C}^{\ast}$. For all
$0 \leq p \leq \varsigma - 1$, let $G_{p}$ be a holomorphic function from $\mathcal{E}_{p}$ into
the Banach space $(\mathbb{F},||.||_{\mathbb{F}})$ and let the cocycle $\Theta_{p}(\epsilon) = G_{p+1}(\epsilon) - G_{p}(\epsilon)$
be a holomorphic function from the sector $Z_{p} = \mathcal{E}_{p+1} \cap \mathcal{E}_{p}$ into $\mathbb{E}$
(with the convention that $\mathcal{E}_{\varsigma} = \mathcal{E}_{0}$ and $G_{\varsigma} = G_{0}$).
We make the following assumptions.\medskip

\noindent {\bf 1)} The functions $G_{p}(\epsilon)$ are bounded as $\epsilon \in \mathcal{E}_{p}$ tends to the origin
in $\mathbb{C}$, for all $0 \leq p \leq \varsigma - 1$.\medskip

\noindent {\bf 2)} The functions $\Theta_{p}(\epsilon)$ are exponentially flat of order $k$ on $Z_{p}$, for all
$0 \leq p \leq \varsigma-1$. This means that there exist constants $C_{p},A_{p}>0$ such that
$$ ||\Theta_{p}(\epsilon)||_{\mathbb{F}} \leq C_{p}e^{-A_{p}/|\epsilon|^{k}} $$
for all $\epsilon \in Z_{p}$, all $0 \leq p \leq \varsigma-1$.\medskip

Then, for all $0 \leq p \leq \varsigma - 1$, the functions $G_{p}(\epsilon)$ have a common formal power series
$\hat{G}(\epsilon) \in \mathbb{F}[[\epsilon]]$ as Gevrey asymptotic expansion of order $1/k$ on $\mathcal{E}_{p}$. Moreover,
if the aperture of one sector $\mathcal{E}_{p_0}$ is slightly larger than $\pi/k$, then $G_{p_0}(\epsilon)$ represents
the $k-$sum of $\hat{G}(\epsilon)$ on $\mathcal{E}_{p_0}$.}

\subsection{Gevrey asymptotic expansion in the complex parameter for the analytic solutions to the initial value problem}

Within this subsection, we disclose the second central result of our work, namely we establish the existence of a formal power series
in the parameter $\epsilon$ whose coefficients are bounded holomorphic
functions on the product of a sector $\mathcal{T}$ with small radius centered at 0 and a strip $H_{\beta'}$ in $\mathbb{C}^2$, which
represent the common Gevrey asymptotic expansion of order $1/k$ of the actual solutions
$u_{p}(t,z,\epsilon)$ of (\ref{main_ivp_u}) constructed in Theorem 1.\medskip

\noindent The second main result of this work can be stated as follows.

\begin{theo} We set $\mathbb{F}$ as the Banach space of complex valued bounded holomorphic functions on the product
$(\mathcal{T} \cap D(0,\sigma')) \times H_{\beta'}$ endowed with the supremum norm where the sector $\mathcal{T}$, radius $\sigma'>0$ and
width $\beta'>0$ are determined in Theorem 1. For all $0 \leq p \leq \varsigma - 1$, the holomorphic and bounded functions
$\epsilon \mapsto u_{p}(t,z,\epsilon)$
from $\mathcal{E}_{p}$ into $\mathbb{F}$ built up in Theorem 1 possess a formal power series
$$ \hat{u}(t,z,\epsilon) = \sum_{m \geq 0} h_{m}(t,z) \epsilon^{m} \in \mathbb{F}[[\epsilon ]]$$
as Gevrey asymptotic expansion of order $1/k$. Strictly speaking, for all $0 \leq p \leq \varsigma-1$, we can pick up two constants
$C_{p},M_{p}>0$ with
$$ \sup_{t \in \mathcal{T} \cap D(0,\sigma'),z \in H_{\beta'}} |u_{p}(t,z,\epsilon) - \sum_{m=0}^{n-1} h_{m}(t,z) \epsilon^{m}|
\leq C_{p}M_{p}^{n}\Gamma(1 + \frac{n}{k}) |\epsilon|^{n}
$$
for all $n \geq 1$, whenever $\epsilon \in \mathcal{E}_{p}$. Furthermore, if the aperture of one sector $\mathcal{E}_{p_0}$ can be taken
slightly larger than
$\pi/k$, then the map $\epsilon \mapsto u_{p_0}(t,z,\epsilon)$ is promoted as the $k-$sum of $\hat{u}(t,z,\epsilon)$ on $\mathcal{E}_{p_0}$.
\end{theo}
\begin{proof} We focus on the family of functions $u_{p}(t,z,\epsilon)$, $0 \leq p \leq \varsigma-1$ constructed in Theorem 1.
For all $0 \leq p \leq \varsigma-1$, we define $G_{p}(\epsilon) := (t,z) \mapsto u_{p}(t,z,\epsilon)$, which represents by construction a
holomorphic and bounded function from $\mathcal{E}_{p}$ into the Banach space $\mathbb{F}$ of bounded holomorphic functions on
$(\mathcal{T} \cap D(0,\sigma')) \times H_{\beta'}$ equipped with the supremum norm, where $\mathcal{T}$ is a bounded sector selected in Theorem 1,
the radius $\sigma'>0$ is taken small enough and $H_{\beta'}$ is a horizontal strip of width $0 < \beta' < \beta$. In accordance with the
bounds (\ref{exp_small_difference_u_p}), we deduce that the cocycle
$\Theta_{p}(\epsilon) = G_{p+1}(\epsilon) - G_{p}(\epsilon)$ is exponentially flat of order $k$ on
$Z_{p} = \mathcal{E}_{p} \cap \mathcal{E}_{p+1}$, for any $0 \leq p \leq \varsigma-1$.

Owing to Theorem (RS) displayed overhead, we obtain a formal power series $\hat{G}(\epsilon) \in \mathbb{F}[[\epsilon]]$
which represents the Gevrey asymptotic expansion of order $1/k$ of each $G_{p}(\epsilon)$ on $\mathcal{E}_{p}$, for $0 \leq p \leq \varsigma - 1$.
Besides, when the aperture of one sector $\mathcal{E}_{p_0}$ is slightly larger than $\pi/k$, the function $G_{p_0}(\epsilon)$ defines
the $k-$sum of $\hat{G}(\epsilon)$ on $\mathcal{E}_{p_0}$ as described through Definition 7.
\end{proof}

\noindent \textbf{Acknowledgements.}  A. Lastra and S. Malek are supported by the Spanish Ministerio de Econom\'{\i}a, Industria y Competitividad under the Project MTM2016-77642-C2-1-P.

\end{document}